\documentclass[a4paper,10pt]{article}
\usepackage[english,frenchb]{babel}
\usepackage[T1]{fontenc}
\usepackage[utf8]{inputenc}
\usepackage{amsmath}
\usepackage{array}
\usepackage{amsthm}
\usepackage{amssymb}
\input xy
\xyoption{all}
\usepackage{hyperref}
\usepackage{stmaryrd}

\usepackage{tikz}
\usetikzlibrary{matrix,arrows,decorations.pathmorphing}
\usepackage{tikz-cd}

\usepackage{enumerate}

\newcommand{\A}{{\mathcal{A}}}
\newcommand{\B}{{\mathcal{B}}}
\newcommand{\C}{{\mathcal{C}}}
\newcommand{\D}{{\mathcal{D}}}
\newcommand{\F}{{\mathcal{F}}}
\newcommand{\fct}{{\mathbf{Fct}}}
\newcommand{\Pp}{{\mathcal{P}}}
\newcommand{\pol}{{\mathcal{P}ol}}
\newcommand{\add}{\mathbf{Add}}
\newcommand{\E}{{\mathcal{E}}}

\newcommand{\s}{{\mathcal{S}}}
\newcommand{\I}{{\mathcal{I}}}
\newcommand{\M}{{\mathcal{M}}}
\newcommand{\G}{{\mathcal{G}}}
\newcommand{\col}{{\rm colim}\,}
\newcommand{\Md}{\text{-}\mathbf{Mod}}
\newcommand{\md}{\text{-}\mathbf{mod}}
\newcommand{\Mdd}{\mathbf{Mod}\text{-}}
\newcommand{\Si}{\mathfrak{S}}
\newcommand{\FF}{\mathbb{F}}
\newcommand{\op}{{\mathrm{op}}}
\newcommand{\PP}{\mathcal{P}}

\newcommand{\GL}{\operatorname{GL}}
\newcommand{\tor}{\mathrm{tor}}
\newcommand{\psf}{PSF}
\newcommand{\ppsf}{\mathbf{psf}}
\newcommand{\pf}{\mathbf{pf}}
\newcommand{\Tor}{\mathrm{Tor}}
\newcommand{\Hom}{\mathrm{Hom}}

\newtheorem{thi}{Th\'eor\`eme}
\newtheorem{pri}[thi]{Proposition}

\newtheorem{thm}{Th\'eor\`eme}[section]
\newtheorem{pr}[thm]{Proposition}
\newtheorem{cor}[thm]{Corollaire}
\newtheorem{lm}[thm]{Lemme}

\theoremstyle{definition}
\newtheorem{defi}[thm]{D\'efinition}
\newtheorem{defiprop}[thm]{Proposition et d\'efinition}
\newtheorem{nota}[thm]{Notation}
\newtheorem{ex}[thm]{Exemple}
\newtheorem{hyp}[thm]{Hypothèse}

\theoremstyle{remark}
\newtheorem{rem}[thm]{Remarque}

\title{Finitude homologique des foncteurs sur une catégorie additive et applications 

\emph{\Large{D\'edi\'e \`a S. Betley, T. Pirashvili et L. Schwartz pour leurs contributions pionni\`eres au sujet}}}

\author{Aur\'elien Djament et Antoine Touz\'e\thanks{Univ. Lille, CNRS, UMR 8524 - Laboratoire Paul Painlev\'e, F-59000 Lille, France.}}

\date{}

\begin{document}

\maketitle

\begin{abstract}
On donne des conditions suffisantes pour qu'un foncteur de longueur finie d'une catégorie additive vers des espaces vectoriels de dimensions finies possède une résolution projective dont les termes sont de type fini. Pour les foncteurs {\em polynomiaux}, on étudie également une propriété de finitude homologique plus faible, qui s'applique à la stabilité homologique à coefficients tordus des monoïdes de matrices. Ces résultats s'inspirent de travaux de Schwartz et Betley-Pirashvili, qu'ils généralisent, et utilisent les théorèmes de décomposition à la Steinberg que nous avons récemment obtenus avec Vespa. Nous montrons également, en guise d'application, une propriété de finitude pour l'homologie stable de groupes linéaires sur des anneaux appropriés.
\end{abstract}

{\selectlanguage{english}
{\begin{abstract}
We give sufficient conditions which ensure that a functor of finite length from an additive category to finite-dimensional vector spaces has a projective resolution whose terms are finitely generated. For {\em polynomial} functors, we study also a weaker homological finiteness property, which applies to twisted homological stability for matrix monoids. This is inspired by works by Schwartz and Betley-Pirashvili, which are generalised; this also uses decompositions {\em à la Steinberg} over an additive category that we recently obtained with Vespa. We show also, as an application, a finiteness property for stable homology of linear groups on suitable rings.
\end{abstract}
}}

\noindent
{\em Mots clefs :} présentation finie ; résolutions ; groupes d'extensions ; foncteurs polynomiaux ; support d'un foncteur ; catégories additives.

\medskip

\noindent
{\em Classification MSC 2020 :} 18A25, 18E10, 18G15, 20J06, 20M05 ; 18A40, 18E05, 18G31.

\section*{Introduction}

\subsection*{Les propriétés $pf_n$ et $psf_n$}

Dans une catégorie abélienne engendrée par des objets projectifs de type fini $P_i$, un objet $X$ est dit de $n$-présentation finie, ou plus brièvement $pf_n$, s'il existe une suite exacte
\[ X_n\to X_{n-1}\to \dots\to X_1\to X_0\to X\to 0\]
dont chaque terme $X_k$ est une somme directe finie de copies des $P_i$. 
Une notion plus faible (qui dépend du choix de la collection de générateurs projectifs de type fini $P_i$) est celle d'objet $psf_n$, ou de $n$-présentation de support fini. Un objet $X$ est $psf_n$ si les termes $X_k$ du complexe exact sont sommes directes de copies (éventuellement en nombre infini) d'un nombre fini de $P_i$. 
L'étude de la propriété $pf_n$ est un thème classique en algèbre commutative, en théorie des représentations \cite{Aus65,Aus82} ou en cohomologie des groupes \cite[chapitre~VIII]{Brown}. La propriété $psf_n$ apparaît par exemple en relation avec les recollements de catégories abéliennes \cite{Psarou}.

Les propriétés $pf_n$ et $psf_n$ sont sans mystère dans les catégories localement noethériennes, telle la catégorie des modules sur un anneau noethérien :  un objet est $pf_\infty$ (c'est-à-dire $pf_n$ pour tout entier $n$) si et seulement s'il est de type fini (c'est-à-dire $pf_0$). En revanche, ces propriétés sont généralement difficiles à établir en l'absence de noethérianité locale. Par exemple, le caractère $pf_\infty$ de la représentation triviale du groupe $\mathrm{Out}(F_n)$ est un théorème remarquable de Culler-Vogtmann \cite[corollaire~6.1.3]{CuVo} faisant appel à des techniques géométriques.

\subsection*{Les propriétés $pf_n$ et $psf_n$ pour les foncteurs sur une catégorie additive} Dans le présent article, nous étudions les propriétés $pf_n$ et $psf_n$ dans les catégories de foncteurs (non nécessairement additifs) sur une petite catégorie additive --- par exemple, la catégorie $\F(A,k)$ des foncteurs  de source les $A$-modules projectifs de type fini sur un anneau $A$, et de but les espaces vectoriels sur un corps commutatif $k$. 
Ces catégories de foncteurs sont rarement localement noethériennes \cite[proposition~11.1]{DTV}.
% et la noethérianité locale est difficile à établir dans les quelques cas où l'on peut la démontrer \cite{PSam,SamSn}. 
L'objectif principal de l'article est d'établir des critères pour les propriétés $pf_n$ et $psf_n$, qui soient facilement exploitables en pratique, et valables en l'absence de noethérianité locale. Nous en donnons aussi quelques applications.

Le cas classiquement bien compris est celui de la catégorie $\F(\mathbb{F}_q,k)$ où $\mathbb{F}_q$ est un corps fini. Le \emph{lemme de Schwartz} \cite[§\,10]{FLS} établit que tous les foncteurs de type fini et \emph{polynomiaux} à la Eilenberg-Mac Lane \cite{EML} possèdent la propriété $pf_\infty$. Plus récemment, les travaux de Putman, Sam et Snowden \cite{PSam,SamSn} ont établi la noethérianité locale de $\F(\mathbb{F}_q,k)$, montrant que tous les foncteurs de type fini possèdent la propriété $pf_\infty$.

Une première partie de nos résultats s'inspire du lemme de Schwartz et l'étend à des catégories de foncteurs de source additive assez générale, où la noethérianité locale n'est pas satisfaite. Ces résultats reposent notamment sur la structure des foncteurs simples polynomiaux mise à jour avec Vespa dans \cite[\S, 5.1]{DTV}. 
Une deuxième partie de nos résultats s'affranchit de l'hypothèse de polynomialité, et établit la propriété $pf_n$ pour des foncteurs non nécessairement polynomiaux. Ces résultats reposent sur un autre résultat de structure de \cite[§\,4.2]{DTV}, qui permet décomposer un foncteur simple en une partie polynomiale qui relève des généralisations des lemmes de Schwartz établies précédemment,  et une partie \emph{antipolynomiale} qui peut être abordée directement.

\subsection*{Résultats principaux} Les théorèmes suivants sont obtenus comme conséquences d'énoncés plus généraux, notamment avec des catégories source ou but plus générales. Dans les énoncés suivants, $A$ désigne un anneau, $k$ un corps commutatif.

\begin{thi}[Corollaire~\ref{cor-ttnoeth} dans le corps de l'article]\label{thi1b} Si les anneaux $A$ et $A\otimes_\mathbb{Z}k$ sont noethériens à droite, alors tout foncteur de $\F(A,k)$ de longueur finie et à valeurs de dimensions finies  vérifie la propriété $pf_\infty$.
\end{thi}

\begin{thi}[Théorème~\ref{cor-psf} dans le corps de l'article]\label{thi1a} Tout foncteur polynomial de $\F(A,k)$ vérifie la propriété $psf_\infty$.
\end{thi}

\subsection*{Applications de la propriété $pf_n$}

Dans un cas très général, la propriété $pf_n$ possède des applications importantes : finitude homologique (proposition~\ref{pr-valext-fini}), formules de K\"unneth en cohomologie (proposition~\ref{pr-kunneth}). Dans \cite{DT}, nous montrons que de nombreux groupes de torsion ou d'extensions entre foncteurs simples sur une catégorie additive sont accessibles au calcul. Les résultats pour les groupes d'extensions nécessitent en général plus d'hypothèses de finitude que ceux concernant les groupes de torsion --- voir notamment \cite[§\,6.3 conclusion~6.7 et §\,10.4]{DT}. La propriété $pf_\infty$ joue ainsi rôle important dans l'étude des groupes d'extensions entre foncteurs sur des petites catégories additives générales.

L'homologie des foncteurs polynomiaux sur une catégorie additive est également liée, par un théorème fondamental de Scorichenko \cite{Sco}, à l'homologie stable des groupes linéaires. Ainsi, des propriétés de finitude en homologie des foncteurs qu'on peut déduire de nos résultats possèdent des applications à l'homologie des groupes linéaires. En nous fondant sur un énoncé analogue au théorème~\ref{thi1b} (le théorème~\ref{th-polnlpf}) valable pour les foncteurs à valeurs dans les groupes abéliens, nous montrons la proposition~\ref{pri2} ci-dessous. Précisons avant de l'énoncer nos notations : si $A$ est un anneau, $\GL(A)$ désigne le groupe linéaire stable $\underset{i\in\mathbb{N}}{\col}\GL_i(A)$, et l'on note (HGLF) la propriété suivante :

\begin{itemize}
\item[(HGLF)] Pour tout $n\in\mathbb{N}$, le groupe abélien $H_n(\GL(A);\mathbb{Z})$ est de type fini.
\end{itemize}

\begin{pri}[Corollaire~\ref{cor-HGLF-ideal} dans le corps de l'article]\label{pri2} Soient $A$ un anneau dont le groupe additif est de type fini et $I$ un idéal bilatère nilpotent de $A$. Si $A/I$ vérifie la propriété {\rm (HGLF)}, alors il en est de même pour $A$.
\end{pri}

\subsection*{Applications de la propriété $psf_n$}

Nous donnons au théorème~\ref{cor-psf} une version quantitative plus précise du théorème~\ref{thi1a}, qui nous permet d'améliorer la borne de stabilité exponentielle de Betley-Pirashvili \cite{BP} en une borne linéaire pour l'homologie des monoïdes de matrices sur un anneau $A$ arbitraire à coefficients tordus (qui est reliée directement à celle de la catégorie $\mathbf{P}(A)$) :

\begin{pri}[Cf. proposition~\ref{pr-BP1}]\label{pri1} Soient $A$ un anneau, $B : \mathbf{P}(A)^\op\times\mathbf{P}(A)\to\mathbf{Ab}$ un bifoncteur et $d,n\in\mathbb{N}$. Supposons que $B$ est polynomial de degré au plus $d$ par rapport à l'une des variables. Alors le morphisme de stabilisation
$$HH_i(\mathbb{Z}[\M_n(A)];B(A^n,A^n))\to HH_i(\mathbb{Z}[\M_{n+1}(A)];B(A^{n+1},A^{n+1}))$$
est un isomorphisme pour $n\ge d.(i+2)$ et un épimorphisme pour $n\ge d.(i+2)-1$.
\end{pri}

Signalons que la propriété $psf_n$ et ses renforcements quantitatifs (tels que celui mentionné ci-dessus) ont fait l'objet de travaux par différents auteurs dans des catégories de foncteurs de source non additive. Ainsi, dans \cite[§\,5]{BP},  Betley et Pirashvili obtiennent un résultat de stabilité homologique pour les monoïdes d'endomorphismes des groupes libres de rang fini (similaire à la proposition~\ref{pri1}) qui est essentiellement équivalent à la propriété $psf_\infty$ pour les foncteurs polynomiaux sur les groupes libres de rang fini.
Par ailleurs, de nombreux articles récents, tels celui de Church-Ellenberg \cite{CE-HFI}, étudient des versions quantitatives de la propriété $psf_\infty$, avec un point de vue et une terminologie (la régularité de Castelnuovo-Mumford) inspirés de l'algèbre commutative, pour les foncteurs sur les ensembles finis avec injections (appelés aussi $\mathrm{FI}$-modules). Ces résultats possèdent des applications au-delà des catégories de foncteurs, comme \cite[théorème~D]{CE-HFI}, qui concerne l'homologie de groupes de congruence.

\subsection*{Organisation de l'article} Les trois premières sections sont des sections pré\-li\-mi\-naires dont la plupart des résultats sont connus des experts, mais souvent difficiles à trouver dans la littérature. Ils traitent de situations plus générales que des foncteurs sur une catégorie additive.

Précisément, la section~\ref{sct-pfn} introduit les propriétés $pf_n$ dans le cadre des catégories de Grothendieck. Celles-ci ne possèdent généralement pas assez d'objets projectifs ; l'équivalence de la définition générale~\ref{dfpfn} avec celle donnée au début de cette introduction dans le cas où la catégorie est engendrée par des objets projectifs de type fini est donnée à la proposition~\ref{pr-pres-concr}.

La section~\ref{sct-psf} introduit la propriété $psf_n$ et des notions connexes dans le cadre des catégories de foncteurs d'une petite catégorie (non nécessairement additive) vers une catégorie de Grothendieck, où la notion fondamentale de {\em support} d'un foncteur permet de définir les précisions quantitatives de la propriété $psf_n$ qui sous-tendent la proposition~\ref{pri1}. L'équivalence des définitions données dans ce cadre (\ref{dfpsfn} et \ref{dfsup}) avec celle introduite à partir d'une famille de générateurs projectifs de type fini en début d'introduction fait l'objet de la proposition~\ref{pr-psfn-ind}. On discute les liens entre les propriétés $pf_n$ et $psf_n$ au §\,\ref{ssct-psfn}, et l'on présente la traduction de versions quantitatives de la propriété $psf_n$ en termes de comparaison de l'homologie des foncteurs et de celles de monoïdes au corollaire~\ref{cor-ext-monoides}.

Dans la section~\ref{s-pfpt}, on donne des résultats généraux sur le comportement des propriétés $pf_n$ et $psf_n$ relativement au produit tensoriel. Ils s'avèrent importants plus tard dans l'article dans la mesure où plusieurs des théorèmes fondamentaux de \cite{DTV} qu'on utilise s'expriment par des décompositions tensorielles.

La section~\ref{s-psf} établit le caractère $psf_\infty$ des foncteurs polynomiaux sur $\mathbf{P}(A)$ (pour un anneau $A$ quelconque), sous la forme quantitative plus précise qui permet d'en déduire la proposition~\ref{pri1}, en suivant la méthode de \cite[§\,10]{FLS}.

La section~\ref{sec-add} est consacrée à la propriété $pf_n$ pour les foncteurs additifs d'une petite catégorie additive vers une catégorie de modules. Il s'agit d'étudier quand l'inclusion de la sous-catégorie pleine des foncteurs additifs dans la catégorie de tous les foncteurs préserve la propriété $pf_n$. En effet, pour $n>1$, la propriété $pf_n$ pour un foncteur additif n'est généralement par équivalente dans les deux catégories. Le caractère $pf_n$ est souvent beaucoup plus simple à établir dans la catégorie des foncteurs additifs, qui bénéficie bien plus fréquemment de propriétés de finitude comme la noethérianité locale, mais c'est le caractère $pf_n$ dans la catégorie de tous les foncteurs qui importe pour les applications telles que les propropositions~\ref{pri1} et~\ref{pri2}.

La section~\ref{schw2} constitue le c\oe ur de l'article. Elle démontre la propriété $pf_n$ pour plusieurs classes de foncteurs polynomiaux d'une petite catégorie additive vers une catégories de modules (sur l'anneau des entiers ou un corps). Elle part des résultats de la section~\ref{sec-add} sur les foncteurs additifs pour les étendre à des foncteurs polynomiaux de degré arbitraire en utilisant des résultats de structure des foncteurs polynomiaux, pour certains dus à Pirashvili, pour d'autres tirés de \cite[§\,5.1]{DTV}. Les résultats de la section~\ref{s-pfpt} y jouent également un rôle important. On y établit aussi la proposition~\ref{pri2}.

La section~\ref{sap}, indépendante des précédentes, est consacrée à la propriété $pf_n$ pour les foncteurs {\em antipolynomiaux}, introduits dans \cite[§\,4.1]{DTV} comme \guillemotleft~brique élémentaire~\guillemotright\ des foncteurs d'une catégorie additive vers des espaces vectoriels de dimension finie, à côté des foncteurs polynomiaux. Elle repose sur des techniques simpliciales à la Dold-Kan-Puppe \cite{DoP}.

La section~\ref{sfinale} combine les résultats des sections~\ref{schw2} et~\ref{sap} pour établir, grâce au théorème de décomposition \cite[théorème~4.12]{DTV}, un critère pour qu'un foncteur de longueur finie d'une catégorie additive vers des espaces vectoriels de dimensions finies possède la propriété $pf_n$. Il possède comme cas particulier fondamental le théorème~\ref{thi1b}.

\subsection*{Remerciements} Une partie significative des résultats de cet article figurait dans des versions préliminaires de \cite{DTV}. Nous sommes très reconnaissants à notre co-autrice de \cite{DTV} Christine Vespa de nous avoir permis de les exposer ici. Nous remercions aussi le ou la rapporteur(e) anonyme pour sa relecture minutieuse qui a permis l'amélioration des premières versions de ce texte.

Les auteurs ont bénéficié du soutien partiel de l’Agence Nationale de la Recherche,
via le projet ANR ChroK (ANR-16-CE40-0003), le Labex CEMPI (ANR-11-LABX-0007-01), et, pour le premier auteur, le projet ANR AlMaRe (ANR-19-CE40-0001-01). Ils ne soutiennent pas pour autant le principe de l’ANR, dont ils revendiquent la restitution
des moyens aux laboratoires sous forme de crédits récurrents.

\subsection*{Conventions}
\begin{itemize}
\item Les anneaux seront toujours supposés unitaires et associatifs.
\item Dans tout cet article, $k$ désigne un anneau commutatif.
Les produits tensoriels de base non spécifiée sont pris sur $k$.
\item 
Si $A$ est un anneau, on note $A\Md$ (resp. $\Mdd A$) la catégorie des $A$-modules à gauche (resp. à droite), et $\mathbf{P}(A)$ la sous-catégorie pleine de $A\Md$ des modules projectifs de type fini.
\item 
Si $\C$ est une catégorie, on note $\C(x,y)$ l'ensemble des morphismes de source $x$ et de but $y$ dans $\C$. De plus, si $t$ est un objet d'une catégorie cocomplète (resp. complète) $\C$ et $E$ un ensemble, on note $t[E]$ (resp. $t^E$) un coproduit (resp. produit) de copies de $t$ indexées par $E$. Ce coproduit (resp. produit) peut être choisi de façon fonctorielle covariante en $t$, et fonctorielle covariante (resp. contravariante) en $E$.
\item 
On nomme simplement {\em petite} une catégorie usuellement appelée {\em essentiellement petite}, c'est-à-dire possédant un squelette formant un ensemble.
Tout au long de l'article, la lettre $\A$ désigne une petite catégorie additive, et les lettres $\C$, $\D$, $\C'$ et $\D'$ désignent des petites catégories.
\item
On désigne par $\E$ et $\E'$ des catégories de Grothendieck, c'est-à-dire des catégories abéliennes cocomplètes, possédant un générateur, et dans lesquelles les colimites filtrantes sont exactes. On rappelle qu'une catégorie de Grothendieck possède toujours des produits \cite[chapitre~3, corollaire~7.10]{Pop} et un cogénérateur injectif \cite[chapitre~3, théorème~10.10 et lemme~7.12]{Pop}. En particulier, on peut y définir des groupes d'extensions.

\item On désigne par $\fct(\C,\E)$ la catégorie des foncteurs de source $\C$ et de but $\E$. C'est une catégorie de Grothendieck. Lorsque $\E=A\Md$, on notera cette catégorie $\F(\C;A)$ au lieu de $\fct(\C,A\Md)$. 

\item Si $\phi:\C_1\to \C_2$ est un foncteur entre petites catégories, on note $\phi^*:\fct(\C_2,\E)\to \fct(\C_1,\E)$ le foncteur de précomposition par $\phi$. C'est un foncteur exact, qui préserve les limites et les colimites.
\item On désigne par $\add(\B,\E)$ la catégorie des foncteurs additifs d'une petite catégorie additive $\B$ dans $\E$. C'est une catégorie de Grothendieck. Lorsque  $\E=A\Md$, on notera cette catégorie $\add(\B;A)$ au lieu de $\add(\B,A\Md)$.
\end{itemize}

\section{La propriété $pf_n$}\label{sct-pfn}

Les objets de présentation finie donnent lieu depuis longtemps à des recherches systématiques, telles celles, fondatrices, d'Auslander \cite{Aus65}. La propriété de présentation finie {\em supérieure}, bien connue, a été étudiée largement dans les catégories de modules, notamment sur les algèbres de groupes \cite[chapitre~VIII, §\,4 et 5]{Brown}. Elle semble toutefois avoir fait l'objet de peu de travaux systématiques dans le cadre général des catégories de Grothendieck, qui ne possèdent pas nécessairement assez d'objets projectifs, avant \cite{BGP}. Nous ferons dans cette section les rappels nécessaires sur cette propriété, avant de nous concentrer sur le cas des catégories de foncteurs.

\subsection{Cas des catégories de Grothendieck}\label{ssct-pfn}

\begin{defi}\label{dfpfn} Soient $n\in\mathbb{N}\cup\{\infty\}$ et $X$ un objet d'une catégorie de Grothendieck $\E$. On dit que $X$ est \emph{de $n$-présentation finie} (en abrégé, $pf_n$) si pour toute petite catégorie filtrante $I$ et tout foncteur $\Phi : I\to\E$, l'application canonique
\begin{equation}\label{eqmc}
\alpha_i(X,\Phi) : \underset{I}{\col}\mathrm{Ext}^i_\E(X,-)\circ\Phi\to\mathrm{Ext}^i_\E(X,\underset{I}{\col}\Phi)
\end{equation}
est bijective pour $i<n$ et injective pour $i<n+1$.

On note $\pf_n(\E)$ la classe des objets de $\E$ vérifiant la propriété $pf_n$.
\end{defi}

La terminologie est motivée par le cas particulier où la catégorie $\E$ possède assez d'objets projectifs de type fini : la propriété $pf_n$ équivaut alors à l'existence d'une $n$-présentation par des projectifs de type fini, comme nous le verrons dans la proposition~\ref{pr-pres-concr} ci-après.

Pour simplifier certains énoncés, on conviendra que tout objet de $\E$ vérifie $pf_n$ pour $n<0$.

\begin{ex}\label{expfinf}
\begin{enumerate}
\item[(a)] Les objets de $0$-présentation finie sont les objets de type fini. Les objets de $1$-présentation finie sont les objets de présentation finie.

Pour la définition classique des notions d'objets de type (resp. présentation) fini(e) d'une catégorie de Grothendieck et leurs premières propriétés, on pourra consulter \cite[§\,3.5, pages 90-92]{Pop}. L'équivalence entre la propriété de présentation finie et $pf_1$ découle par exemple de \cite[\emph{Theorem~5.10}]{Pop} et de la proposition~\ref{prpfnel} ci-après.
\item[(b)] Tout objet projectif de type fini de $\E$ appartient à $\pf_\infty(\E)$.
\end{enumerate}
\end{ex}

L'énoncé classique suivant est le théorème~4.7 de \cite{BGP}.

\begin{pr}\label{anc-rqc} Si $\E$ est localement noethérienne, tout objet de type fini de $\E$ possède la propriété $pf_\infty$.
\end{pr}

Les propriétés suivantes sont bien connues ; elles s'obtiennent par des arguments formels d'algèbre homologique. La plupart d'entre elles peuvent se trouver dans \cite{BGP}, par exemple. Nous en omettrons donc la démonstration. Nous indiquons toutefois comment établir rapidement l'énoncé suivant, pour lequel \cite[lemme~2.6]{BGP} suppose une hypothèse supplémentaire superflue ($\E$ localement de présentation finie).

\begin{pr}\label{prpfnel} Soient $n\in\mathbb{N}^*$ et $X$ un objet de $\E$. Les assertions suivantes sont équivalentes :
\begin{enumerate}
\item\label{pfe1} $X$ appartient à $\pf_n(\E)$ ;
\item\label{pfe2} pour toute petite catégorie filtrante $I$ et tout foncteur $\Phi : I\to\E$, l'application $\alpha_i(X,\Phi)$ est bijective si $i<n$ ;
\item\label{pfe3} pour toute petite catégorie filtrante $I$ et tout foncteur $\Psi : I\to\E$ à valeurs injectives dans $\E$, l'application $\alpha_i(X,\Psi)$ est bijective si $i<n$.
\end{enumerate}
\end{pr}

\begin{proof} Les implications (\ref{pfe1})$\Rightarrow$(\ref{pfe2})$\Rightarrow$(\ref{pfe3}) sont évidentes ; montrons (\ref{pfe3})$\Rightarrow$(\ref{pfe1}). Supposons donc (\ref{pfe3}) satisfaite, et considérons un foncteur $\Phi : I\to\E$, où $I$ est une petite catégorie filtrante. Alors $\Phi$ peut se plonger dans un foncteur $\Psi$ à valeurs injectives --- par exemple, $M^{\E(-,M)\circ\Phi}$, où $M$ est un cogénérateur injectif de $\E$. On a un diagramme commutatif aux colonnes exactes
$$\xymatrix{\vdots\ar[d] & & \vdots\ar[d] \\
\underset{I}{\col}\mathrm{Ext}_\E^{i-1}(X,-)\circ\Psi\ar[d]\ar[rr]^-{\alpha_{i-1}(X,\Psi)} & & \mathrm{Ext}_\E^{i-1}(X,\underset{I}{\col}\Psi)\ar[d] \\
\underset{I}{\col}\mathrm{Ext}_\E^{i-1}(X,-)\circ\Psi/\Phi\ar[d]\ar[rr]^-{\alpha_{i-1}(X,\Psi/\Phi)} & & \mathrm{Ext}_\E^{i-1}(X,\underset{I}{\col}\Psi/\Phi)\ar[d] \\
\underset{I}{\col}\mathrm{Ext}_\E^i(X,-)\circ\Phi\ar[d]\ar[rr]^-{\alpha_i(X,\Phi)} & & \mathrm{Ext}_\E^i(X,\underset{I}{\col}\Phi)\ar[d] \\
\underset{I}{\col}\mathrm{Ext}_\E^i(X,-)\circ\Psi\ar[rr]^-{\alpha_i(X,\Psi)}\ar[d] & & \mathrm{Ext}_\E^i(X,\underset{I}{\col}\Psi)\ar[d]\\
\vdots && \vdots
}$$
qui permet d'obtenir la propriété souhaitée par récurrence, à l'aide d'un argument de décalage, compte-tenu de la nullité de $\underset{I}{\col}\mathrm{Ext}_\E^i(X,-)\circ\Psi$ pour $i>0$, $\Psi$ étant à valeurs injectives.
\end{proof}

\begin{pr}{\rm (Cf. \cite[proposition~2.8]{BGP})}\label{pr-pfn-sec} Soient $0\to Y\to X\to Z\to 0$ une suite exacte courte de $\E$ et $n\in\mathbb{N}\cup\{\infty\}$.
\begin{enumerate}
\item Si $Y$ et $Z$ appartiennent à $\pf_n(\E)$, il en est de même pour $X$.
\item Si $X$ appartient à $\pf_n(\E)$ et $Y$ à $\pf_{n-1}(\E)$, alors $Z$ appartient à $\pf_n(\E)$.
\item Si $X$ appartient à $\pf_n(\E)$ et $Z$ à $\pf_{n+1}(\E)$, alors $Y$ appartient à $\pf_n(\E)$.
\item La classe $\pf_n(\E)$ est stable par facteur direct et par somme directe finie.
\end{enumerate}
\end{pr}

\begin{cor}\label{cor-pfn-compl} Soient $C_\bullet$ un complexe de chaînes de $\E$, $H_\bullet$ son homologie et $n\in\mathbb{N}\cup\{\infty\}$. Supposons que les conditions suivantes sont satisfaites :
\begin{enumerate}
\item pour tout $i\in\mathbb{N}$, $C_i$ appartient à $\pf_{n-i}(\E)$ ;
\item pour tout $i\in\mathbb{N}^*$, $H_i$ appartient à $\pf_{n-i-1}(\E)$.
\end{enumerate}
Alors $H_0$ appartient à $\pf_n(\E)$.
\end{cor}

Si $\mathfrak{C}$ est une classe d'objets de $\E$, on appelle {\em $n$-présentation} d'un objet $X$ de $\E$ par des objets de $\mathfrak{C}$ toute suite exacte
$$C_n\to C_{n-1}\to\dots\to C_0\to X\to 0\qquad\text{si }n\in\mathbb{N}\text{ , ou}$$
$$\cdots\to C_i\to C_{i-1}\to\dots\to C_0\to X\to 0\qquad\text{si }n=\infty\,,$$
où les $C_i$ appartiennent à $\mathfrak{C}$.

La terminologie de {\em $n$-présentation finie} provient de la proposition suivante, appliquée souvent lorsque la classe $\G$ est constituée d'objets projectifs.

\begin{pr}{\rm (Cf. \cite[proposition~2.12]{BGP})}\label{pr-pres-concr} Soit $n\in\mathbb{N}\cup\{\infty\}$ tel que $\E$ possède une classe d'objets génératrice $\G$ incluse dans $\pf_n(\E)$. Notons $\G^\oplus$ la classe des objets de $\E$ qui sont somme directe finie d'éléments de $\G$. 
\begin{enumerate}
\item Un objet $X$ de $\E$ appartient à $\pf_n(\E)$ si et seulement s'il existe une suite exacte courte $0\to Y\to G\to X\to 0$ avec $G$ dans $\G^\oplus$ et $Y$ dans $\pf_{n-1}(\E)$.
\item Un objet $X$ de $\E$ appartient à $\pf_n(\E)$ si et seulement s'il possède une $n$-présentation par des objets de $\G^\oplus$.
\end{enumerate}
\end{pr}

Nous terminons cette section par deux énoncés simples concernant l'effet de foncteurs appropriés sur la propriété $pf_n$. Nous dirons qu'un foncteur $\Phi:\E\to \E'$ \emph{détecte les objets de type fini} (resp. $pf_n$) lorsque pour tout $X$ de $\E$, si $\Phi(X)$ est de type fini (resp. $pf_n$), alors $X$ est de type fini (resp. $pf_n$).

La propriété suivante est immédiate mais utile.

\begin{pr}\label{rq-prestf} Tout foncteur exact, fidèle et cocontinu détecte les objets de type fini.
\end{pr}

\begin{pr}\label{pr-pfn-efonc} Soit $n\in\mathbb{N}\cup\{\infty\}$ tel que $\E$ possède une classe d'objets génératrice $\G$ incluse dans $\pf_n(\E)$. Soit $\Phi : \E\to\E'$ un foncteur exact.
\begin{enumerate}
\item Si $\Phi$ envoie $\G$ dans $\pf_n(\E')$, alors $\Phi$ préserve la propriété $pf_i$ pour $i\le n$.
\item Si $\Phi$ envoie $\G$ dans $\pf_n(\E')$ et si $\Phi$ détecte les objets de type fini, alors $\Phi$ détecte les objets $pf_i$ pour tout $i\le n+1$.
\end{enumerate}
\end{pr}

\begin{proof} La première assertion découle directement de la proposition~\ref{pr-pres-concr}. La deuxième se déduit de la même proposition, de l'assertion précédente et d'un argument de décalage permettant de raisonner par récurrence sur $i$.
\end{proof}

\begin{pr}\label{pr-pfnad} Soient $\Phi : \E\to\E'$ un foncteur possédant un adjoint à droite $\Psi$, et $n\in\mathbb{N}\cup\{\infty\}$. On suppose que $\Psi$ commute aux colimites filtrantes et que l'une des deux conditions suivantes est satisfaite :
\begin{enumerate}
\item $n\le 1$ ;
\item $\Phi$ et $\Psi$ sont exacts.
\end{enumerate}
Alors $\Phi$ préserve les objets $pf_n$.
\end{pr}

\begin{proof} Cela découle de la définition et du fait qu'une adjonction entre foncteurs exacts se propage aux groupes d'extensions.
\end{proof}

\subsection{Cas des catégories de foncteurs}

On rappelle que $\C$ désigne une petite catégorie, $\E$ une catégorie de Grothendieck et $\fct(\C,\E)$ la catégorie des foncteurs de $\C$ vers $\E$. C'est une catégorie de Grothendieck ; les suites exactes, limites ou colimites s'y déterminent au but.

Nous examinons maintenant comment la propriété $pf_n$ se transmet de $\E$ à $\fct(\C,\E)$.
Si $M$ est un objet de $\E$, on rappelle que $M[E]$ désigne la somme directe de copies indexées par $E$, qui peut être choisie fonctoriellement en $E$ et $M$. Un avatar du lemme de Yoneda s'exprime ainsi :
\begin{equation}\label{eq-yon}
\fct(\C,\E)(M[\C(t,-)],F)\simeq\E(M,F(t))\;.
\end{equation}
Cet isomorphisme de groupes abéliens est
naturel en les objets $M$ de $\E$, $t$ de $\C$ et $F$ de $\fct(\C,\E)$. Comme l'évaluation en $t$ est exacte et cocontinue, et que $M\mapsto M[\C(t,-)]$ est également exact (les sommes directes sont exactes dans la catégorie de Grothendieck $\E$), on en déduit, d'après la proposition~\ref{pr-pfnad} :
\begin{pr}\label{yoneda-pfn} Soit $n\in\mathbb{N}\cup\{\infty\}$. Si $M$ appartient à $\pf_n(\E)$, alors le foncteur $M[\C(t,-)]$ appartient à $\pf_n(\fct(\C,\E))$ pour tout objet $t$ de $\C$.
\end{pr}

\begin{cor}\label{cor-herite-pfn}
Soit $n\in\mathbb{N}\cup\{\infty\}$ tel que $\E$ possède une classe  d'objets génératrice $\G$ incluse dans $\pf_n(\E)$. Alors $\fct(\C,\E)$ possède une classe d'objets génératrice incluse dans $\pf_n(\fct(\C,\E))$, à savoir la classe des foncteurs $G[\C(t,-)]$, où $G$ appartient à $\G$ et $t$ est un objet quelconque de $\C$.
\end{cor}

Sous les hypothèses du corollaire~\ref{cor-herite-pfn}, la proposition~\ref{pr-pres-concr} affirme qu'un foncteur est de $n$-présentation finie si et seulement s'il admet une $n$-presentation par des sommes directes finies de foncteurs $G[\C(t,-)]$, où $G$ appartient à $\G$ et $t$ est un objet quelconque de $\C$. C'est souvent dans ce cadre que nous nous placerons.

On peut montrer la préservation de la propriété $pf_n$ par restriction de la catégorie source à l'aide de la propriété utile suivante.

\begin{pr}\label{pr-precompf} Soient $\xi : \D\to\C$ un foncteur possédant un adjoint à gauche et $n\in\mathbb{N}\cup\{\infty\}$. Le foncteur $\fct(\C,\E)\to\fct(\D,\E)$ qu'induit $\xi$ par précomposition préserve la propriété $pf_n$.
\end{pr}

\begin{proof} Ce foncteur est adjoint à gauche à la précomposition par l'adjoint à gauche de $\xi$, et tout foncteur de précomposition est exact et cocontinu. On peut donc appliquer la proposition~\ref{pr-pfnad}.
\end{proof}

Un cas particulier important est celui où $\E=k\Md$ pour un anneau commutatif $k$ ; on désigne alors par $\F(\C;k)$ la catégorie $\fct(\C,\E)$. On appelle \emph{projectifs standards} les foncteurs $P^t_\C:=k[\C(t,-)]$. Il suit de l'isomorphisme de Yoneda \eqref{eq-yon} que les projectifs standards $P^t_\C$, où $t$ est un objet quelconque de $\C$, forment une famille génératrice de projectifs de type fini de $\F(\C;k)$. La proposition~\ref{pr-pres-concr} affirme donc qu'un objet de $\F(\C;k)$ est de $n$-présentation finie si et seulement s'il admet une $n$-presentation par des somme directes finies de projectifs standards.

L'un des intérêts fondamentaux de la propriété $pf_n$ dans les catégories $\F(\C;k)$ réside dans la propriété de finitude suivante des groupes d'extensions ou de torsion (pour lesquels on pourra se référer à \cite[§\,C.10]{Lod}). Cette propriété est souvent appliquée en prenant pour $\mathfrak{F}$ la classe des $k$-modules noethériens, ou finis.

\begin{pr}\label{pr-valext-fini} Soient $n\in\mathbb{N}\cup\{\infty\}$, $F$ et $G$ deux foncteurs de $\F(\C;k)$ et $H$ un foncteur de $\F(\C^\op;k)$. Supposons que $F$ appartient à $\pf_n(\F(\C;k))$, et que $G$ et $H$ prennent leurs valeurs dans une classe $\mathfrak{F}$ de $k$-modules stable par somme directe finie et par sous-quotient. Alors les $k$-modules $\mathrm{Ext}^i_{\F(\C;k)}(F,G)$ et $\mathrm{Tor}^\C_i(H,F)$ appartiennent à $\mathfrak{F}$ pour chaque entier $i\le n$.
\end{pr}
\begin{proof}
Comme $F$ est $pf_n$, il admet une résolution projective $\mathrm{P}_\bullet$ dont les termes de degré inférieur à $n$ sont des sommes finies de projectifs standards. L'isomorphisme de Yoneda $\F(\C;k)(P^t_\C,G)\simeq G(t)$ et la stabilité de $\mathfrak{F}$ par somme directe finie montre que les objets du complexe $\F(\C;k)(\mathrm{P}_\bullet,G)$ sont dans $\mathfrak{F}$ jusqu'en degré $n$, et la stabilité de $\mathfrak{F}$ par sous-quotient garantit que son homologie est dans $\mathfrak{F}$ jusqu'au degré $n$. La démonstration du résultat en $\Tor$ est similaire.
\end{proof}

\section{Supports et propriété $psf_n$}\label{sct-psf} 

La notion de {\em support} d'un foncteur (cf. par exemple \cite[§\,2.3]{Dja-FM}) constitue un outil de base pour en ramener l'étude à celle de foncteurs définis sur une sous-catégorie. La notion de support de présentation supérieure (ou de support de $n$-présentation) en est une version dérivée, bien connue des experts. Toutefois, nous n'en connaissons pas de référence explicite, nous en fournissons donc une exposition complète au §\,\ref{ssct-psf}. La notion de support de $n$-présentation nous permet d'introduire au §\,\ref{ssct-psfn} la propriété de finitude $psf_n$ pour les foncteurs, plus faible que la propriété $pf_n$ (corollaire~\ref{cor-pspfn}). L'important corollaire~\ref{cor-pfpsf} précise le lien entre ces deux propriétés, sous une hypothèse de finitude sur la catégorie source.

La présentation donnée dans cette section de la propriété $psf_n$ diffère de la présentation informelle du début de l'introduction de l'article mais est adaptée à des catégories de foncteurs dont le but ne contient pas nécessairement assez de projectifs, et surtout à la mise en évidence de renforcements quantitatifs de $psf_n$ importants pour les applications. La coïncidence des deux points de vue, dans le cas où la catégorie but des foncteurs est une catégorie de modules, découle de la proposition~\ref{pr-psfn-ind}.

\begin{rem} Dans le cas des $\mathrm{FI}$-modules, les propriétés $pf_n$ et $psf_n$ ont été examinées en détail par de nombreux auteurs, à l'aide de méthodes inspirées de la combinatoire, de l'algèbre commutative ou de la théorie des représentations --- voir par exemple Church-Ellenberg \cite{CE-HFI}.
\end{rem}

\begin{nota} Pour tout foncteur $\phi : \C_1\to\C_2$ entre petites catégories, on note $\phi^* : \fct(\C_2,\E)\to\fct(\C_1,\E)$ le foncteur de précomposition par $\phi$. Son adjoint à gauche, l'extension de Kan à gauche le long de $\phi$ \cite[chapitre~X]{ML-cat}, est notée $\phi_! : \fct(\C_1,\E)\to\fct(\C_2,\E)$.
\end{nota}

Dans toute la section \ref{sct-psf}, on se place dans le cadre suivant: $\C$ désigne une petite catégorie, et $\D$ une sous-catégorie pleine de $\C$. On note $\iota_\D:\D\to \C$ ou simplement $\iota:\D\to \C$ le foncteur d'inclusion.

\subsection{Supports de présentation supérieure}\label{ssct-psf}

Le foncteur d'inclusion $\iota:\D\to \C$ étant pleinement fidèle, l'unité d'adjonction $\mathrm{Id}\to\iota^*\iota_!$ est un isomorphisme. Les propriétés de la coünité d'adjonction sont plus subtiles, et conduisent à la notion de support. 
\begin{defi} Soit $F$ un foncteur de $\fct(\C,\E)$. On dit que $\D$ est un {\em support} (resp. un {\em support de présentation}) de $F$ si la coünité $\iota_!\iota^* F\to F$ est un épimorphisme (resp. un isomorphisme). 
\end{defi}

Notons $\D_\mathrm{discr}$ la sous-catégorie discrète avec les mêmes objets que $\D$ et $\gamma:\D_\mathrm{discr}\to \D$ l'inclusion. L'adjonction entre $\gamma_!$ et $\gamma^*$
donne lieu à un cotriple (ou comonade) $\perp$ sur $\fct(\D,\E)$, voir par exemple \cite[§\,8.6]{Wei}. Ce cotriple permet \cite[§\,8.7]{Wei} de définir, pour tout foncteur additif $\alpha: \fct(\D,\E)\to \B$ de but une catégorie abélienne, des foncteurs dérivés $\mathbf{L}_i(\alpha):\fct(\D,\E)\to \B$. Le foncteur $\gamma^*$ ne conserve d'un foncteur que ses valeurs sur les objets et oublie son effet sur les flèches, de sorte que l'algèbre homologique associée à ce cotriple est une algèbre homologique relative qui \guillemotleft~met à l'écart~\guillemotright\ les questions homologiques dans la catégorie but $\E$ pour ne retenir que les phénomènes spécifiquement liés à la structure fonctorielle.

On introduit les supports de présentation supérieure comme une version dérivée des supports de présentation.

\begin{defi}\label{dfsup} Soit $F$ un foncteur de $\fct(\C,\E)$. 
\begin{enumerate}
\item Si $n$ est un entier strictement négatif, on convient que $\D$ est toujours un support de $n$-présentation de $F$.
\item On dit que $\D$ est un \emph{support de $0$-présentation de $F$} si c'en est un support.
\item Si $n$ est un entier strictement positif ou $n=+\infty$, on dit que $\D$ est un \emph{support de $n$-présentation de $F$} si c'en est un support de présentation, et si de plus $\mathbf{L}_i(\iota_!)(\iota^*F)=0$ pour $0<i<n$.
\end{enumerate}
\end{defi}

La terminologie est motivée par le fait que $\D$ est un support de $n$-présentation de $F$ si et seulement si ce foncteur possède une $n$-présentation par des sommes directes d'objets de la forme $X[\C(d,-)]$ où $d$ est un objet de $\D$ et $X$ un objet de $\E$, comme on le verra à la proposition~\ref{pr-psfn-ind} (équivalence (a)$\Leftrightarrow$(b)).

On rappelle \cite[§\,8.6 et 8.7]{Wei} (voir aussi \cite[§\,C.10]{Lod}) que les foncteurs dérivés $\mathbf{L}_*(\iota_!)$ se calculent sur un foncteur $X$ de $\fct(\D,\E)$ comme l'homotopie \cite[définition~8.3.6]{Wei} d'un objet simplicial $N^\D_\bullet(X)$ de la catégorie abélienne $\fct(\C,\E)$ 
donné en degré $n$ par
$$N^\D_n(X):=\underset{x_0\to\dots\to x_n}{\bigoplus}X(x_0)[\C(x_n,-)],$$
où la somme est prise sur les suites de flèches composables d'objets de $\D$. (Cela résulte de la définition et de l'isomorphisme $\iota_!(M[\D(x,-)])\simeq M[\C(x,-)]$ naturel en les objets $x$ de $\D$ et $M$ de $\E$.)

Si $F$ est un foncteur de $\fct(\C,\E)$, on dispose d'un morphisme canonique $N^\D_\bullet(\iota^*F)\to F$, où $F$ est vu au but comme objet simplicial constant.

\begin{pr}\label{pr-psfn-sec} Soient $0\to G\to F\to H\to 0$ une suite exacte courte de $\fct(\C,\E)$ et $n\in\mathbb{N}\cup\{\infty\}$.
\begin{enumerate}
\item Si $\D$ est un support de $n$-présentation de $G$ et $H$, c'est un support de $n$-présentation de $F$.
\item Si $\D$ est un support de $n$-présentation de $F$ et un support de $(n-1)$-présentation de $G$, c'est un support de $n$-présentation de $H$.
\item Si $\D$ est un support de $n$-présentation de $F$ et un support de $(n+1)$-présentation de $H$, c'est un support de $n$-présentation de $G$.
\end{enumerate}
\end{pr}

\begin{proof} Cette propriété découle formellement de la suite exacte longue associée à la suite exacte courte $0\to G\to F\to H\to 0$ en appliquant le foncteur homologique $\mathbf{L}_*(\iota_!)\iota^*$.
\end{proof}

La propriété suivante est immédiate.

\begin{defiprop} Le foncteur $\iota_!\gamma_! : \fct(\D_{\mathrm{discr}},\E)\simeq\E^{\mathrm{Ob}\,\D}\to\fct(\C,\E)$ est donné sur les objets par
$$(M_s)_{s\in\mathrm{Ob}\,\D}\mapsto\bigoplus_{s\in\mathrm{Ob}\,\D}M_s[\C(s,-)].$$

Les foncteurs de $\fct(\C,\E)$ appartenant à l'image essentielle de $\iota_!\gamma_!$ sont dits \emph{induits depuis $\D_{\mathrm{discr}}$}.  
\end{defiprop}

\begin{lm}\label{lm-inftypres}
Tout foncteur induit depuis $\D_{\mathrm{discr}}$ admet $\D$ pour support d'$\infty$-présentation.
\end{lm}
\begin{proof}
%Pour toute famille $(M_s)_{s\in E}$, on note $M:\D_{\mathrm{discr}}\to \E$ le foncteur tel que $M(s)=M_s$ si $s$ est un objet de $E$ et $M(s)=0$ sinon. Alors le foncteur $\gamma_! M$ est isomorphe à $\bigoplus_{s\in E}M_s[\D(s,-)]$ et le foncteur $\iota_!\gamma_!M$ est isomorphe à $\bigoplus_{s\in E}M_s[\C(s,-)]$.
Soit $F\simeq\iota_!\gamma_!M$ un foncteur induit depuis $\D_{\mathrm{discr}}$. Alors $\D$ est un support de présentation de $\D$ car la coünité $\iota_!\iota_*F\to F$ est un rétracte de l'isomorphisme $\iota_!(\gamma_! M)\to \iota_!(\iota^*\iota_!)(\gamma_!M)$ induit par l'unité d'adjonction $\mathrm{Id}\xrightarrow[]{\simeq} \iota^*\iota_!$.    
De plus, le foncteur $\iota^*F\simeq \gamma_!M$ est $\perp$-projectif, d'où l'annulation des $\mathbf{L}_i(\iota_!)(i^*F)$ pour $i>0$, cf. \cite[\S\,8.7]{Wei}. 
\end{proof}
 
On rappelle qu'un morphisme $f$ entre objets simpliciaux d'une catégorie abélienne est dit {\em $n$-connexe} si $\pi_i(f)$ est un isomorphisme pour $i<n$ et un épimorphisme pour $i<n+1$.

\begin{pr}\label{pr-psfn-ind} Soient $F$ un foncteur de $\fct(\C,\E)$ et $n\in\mathbb{N}\cup\{\infty\}$.
\begin{enumerate}
\item\label{itpsf2}
Les assertions suivantes sont équivalentes : 
\begin{enumerate}
\item[(a)] $\D$ est un support de $n$-présentation de $F$ ;
\item[(b)] $F$ possède une $n$-présentation par des foncteurs induits depuis $\D_\mathrm{discr}$ ;
\item[(c)] le morphisme simplicial canonique $N^\D_\bullet(\iota^*F)\to F$ est $n$-connexe.
\end{enumerate}
\item Si $\D$ est un support de $n$-présentation de $F$, il en est de même pour toute sous-catégorie pleine de $\C$ contenant $\D$.
\end{enumerate}
\end{pr}

\begin{proof}
L'équivalence entre (a) et (c) résulte de la définition, et l'implication (c)$\Rightarrow$(b) est triviale.
L'implication (b)$\Rightarrow$(a) découle du lemme \ref{lm-inftypres} et d'une application itérée de la proposition~\ref{pr-psfn-sec}.

La deuxième assertion provient de l'équivalence (a)$\Leftrightarrow$(b) et de ce qu'un foncteur induit depuis $\D_{\mathrm{discr}}$ est induit depuis toute sous-catégorie discrète de $\C$ contenant $\D_{\mathrm{discr}}$.
\end{proof}

\begin{pr}\label{pr-psfn-stab} Soient $F$ un foncteur de $\fct(\C,\E)$ et $n\in\mathbb{N}\cup\{\infty\}$. Les assertions suivantes sont équivalentes :
\begin{enumerate}
\item\label{hpsf1} $\D$ est un support de $n$-présentation de $F$ ;
\item\label{hpsf2} pour tout foncteur $G$ de $\fct(\C,\E)$, l'application canonique
$$\mathrm{Ext}^i_{\fct(\C,\E)}(F,G)\to\mathrm{Ext}^i_{\fct(\D,\E)}(\iota^*F,\iota^*G)$$
est bijective pour $i<n$ et injective pour $i<n+1$.
\end{enumerate}
Si $n>0$, elles équivalent également à :
\begin{enumerate}
\item[(3)] pour tout foncteur $G$ de $\fct(\C,\E)$, l'application canonique
$$\mathrm{Ext}^i_{\fct(\C,\E)}(F,G)\to\mathrm{Ext}^i_{\fct(\D,\E)}(\iota^*F,\iota^*G)$$
est bijective pour $i<n$.
\end{enumerate}
\end{pr}

\begin{proof} L'équivalence de (\ref{hpsf1}) et (\ref{hpsf2}) pour $n= 0$ résulte de la définition. On se fixe donc un entier $n>0$. 

Soit $M$ un cogénérateur injectif de $\E$. Considérons la propriété
\begin{enumerate}
\item[(4)] Pour toute famille d'objets $(t_j)$ de $\C$, le morphisme canonique
$$\fct(\C,\E)(F,\prod_j M^{\C(-,t_j)})\to\fct(\D,\E)(\iota^*F,\iota^*\prod_j M^{\C(-,t_j)})$$
est bijectif et $\mathrm{Ext}^i_{\fct(\D,\E)}(\iota^*F,\iota^*\prod_j M^{\C(-,t_j)})=0$ pour $0<i<n$.
\end{enumerate}

Comme les foncteurs de la forme $\prod_j M^{\C(-,t_j)}$ sont injectifs  et que tout foncteur de $\fct(\C,\E)$ se plonge dans un foncteur de ce type, un argument formel de corésolution de $G$ fournit l'implication (4)$\Rightarrow$(\ref{hpsf2}), et (3)$\Rightarrow$(4). Comme (\ref{hpsf2})$\Rightarrow$(3) évidemment, il suffit de montrer l'équivalence de (\ref{hpsf1}) et (4). Celle-ci découle des isomorphismes naturels
$$\E\Big(\bigoplus_j\mathbf{L}_i(\iota_!)(X)(t_j),M\Big)\simeq\mathrm{Ext}^i_{\fct(\D,\E)}\Big(X,\iota^*\prod_j M^{\C(-,t_j)}\Big)$$
qui se déduisent de la propriété universelle \cite[8.7.4]{Wei}.
\end{proof}

Dans le cas où $\D$ est réduite à un objet, ces notions permettent de comparer la cohomologie des foncteurs à celle des monoïdes :

\begin{cor}\label{cor-ext-monoides} Soient $n\in\mathbb{N}\cup\{\infty\}$, $F$ un foncteur de $\F(\C;k)$ et $t$ un objet de $\C$. Les assertions suivantes sont équivalentes :
\begin{enumerate}
\item La sous-catégorie pleine de $\C$ ayant $t$ pour unique objet est un support de $n$-présentation de $F$ ;
\item pour tout foncteur $G$ de $\F(\C;k)$, l'application canonique
$$\mathrm{Ext}^i_{\F(\C;k)}(F,G)\to\mathrm{Ext}^i_{k[\mathrm{End}(t)]}(F(t),G(t))$$
est bijective pour $i<n$ et injective pour $i<n+1$.
\end{enumerate}
\end{cor}

\subsection{La propriété $psf_n$}\label{ssct-psfn}

\begin{defi}\label{dfpsfn} Soient $n$ un entier et $F$ un foncteur de $\fct(\C,\E)$. On dit que $F$ est {\em à $n$-présentation de support fini} (ou simplement à support fini, si $n=0$) --- ou, plus brièvement, possède la propriété $psf_n$ --- s'il existe une sous-catégorie pleine de $\C$ ayant un nombre fini d'objets qui en constitue un support de $n$-présentation. On dit que $F$ vérifie la propriété $psf_\infty$ s'il vérifie $psf_n$ pour tout $n\in\mathbb{N}$.

On note $\ppsf_n(\fct(\C,\E))$ la classe des foncteurs de $\fct(\C,\E)$ possédant la propriété $psf_n$.
\end{defi}

\begin{rem} On prendra garde que, si $F$ est un foncteur $F$ de $\fct(\C,\E)$ vérifiant la propriété $psf_\infty$, il n'existe pas nécessairement de sous-catégorie pleine de $\C$ avec un nombre fini d'objets en constituant un support d'$\infty$-présentation. Par exemple, dans la catégorie $\F(\mathbf{P}(\FF_2);\FF_2)$, qui est localement noethérienne par Putman-Sam-Snowden \cite{PSam,SamSn}, tout foncteur de type fini vérifie $pf_\infty$, et donc $psf_\infty$ (cf. corollaire~\ref{cor-pspfn}). Mais les foncteurs de type fini de $\F(\mathbf{P}(\FF_2);\FF_2)$ possédant un support d'$\infty$-présentation fini sont beaucoup moins nombreux : Powell \cite[théorème~5.0.1]{GP-cow} en a donné une caractérisation précise qui implique par exemple qu'aucun foncteur simple non constant de $\F(\mathbf{P}(\FF_2);\FF_2)$ n'a cette propriété. Un phénomène analogue advient dans la catégorie $\F(\mathrm{FI};k)$, où $k$ un anneau noethérien, par \cite[théorème~A.9]{Dja-FM}.
\end{rem}

Dans l'énoncé suivant, nous disons qu'un foncteur $\Phi$ d'une catégorie filtrante $I$ vers $\fct(\C,\E)$ est {\em ponctuellement stationnaire} si pour tout objet $t$ de $\C$, il existe un objet $a$ de $I$ tel que $\Phi$ envoie toute flèche de source $a$ sur un morphisme qui évalué en $t$ est un isomorphisme de $\E$.

\begin{pr}\label{pr-psfcol} Soient $F$ un foncteur de $\fct(\C,\E)$ et $n\in\mathbb{N}\cup\{\infty\}$. Les assertions suivantes sont équivalentes :
\begin{enumerate}
\item\label{itpsc1} $F$ appartient à $\ppsf_n(\fct(\C,\E))$ ;
\item\label{itpsc2} pour toute petite catégorie filtrante $I$ et tout foncteur ponctuellement stationnaire $\Phi : I\to\fct(\C,\E)$, l'application canonique
$$\alpha_i(F,\Phi) : \underset{I}{\col}\mathrm{Ext}^i_{\fct(\C,\E)}(F,-)\circ\Phi\to\mathrm{Ext}^i_{\fct(\C,\E)}(F,\underset{I}{\col}\Phi)$$
est bijective pour $i<n$ et injective pour $i<n+1$.
\end{enumerate}
\end{pr}

\begin{proof} Disons que $F$ est $psf'_n$ s'il vérifie la condition (\ref{itpsc2}). L'isomorphisme en Ext déduit de l'adjonction \eqref{eq-yon} (page~\pageref{eq-yon}) montre que tout foncteur induit depuis une classe finie d'objets de $\C$  est $psf'_\infty$.

Si $\D$ est une sous-catégorie pleine de $\C$, notons $F_\D$ l'image de la coünité $(\iota_\D)_!\iota_\D^*F\to F$ (où $\iota_D$ a la signification de l'hypothèse~\ref{ssct-psf}). Alors $F_\D$ est le plus petit sous-foncteur de $F$ qui coïncide avec celui-ci évalué sur $\D$. Ainsi, le foncteur depuis l'ensemble ordonné (par inclusion) filtrant des sous-ensembles finis d'un squelette de $\mathrm{Ob}\,\C$ vers $\fct(\C,\E)$ envoyant $\D$ sur $F/F_\D$ est ponctuellement stationnaire, et sa colimite est nulle. Par conséquent, si $F$ est $psf'_0$, il existe un ensemble fini $\D$ d'objets de $\C$ tel que le morphisme canonique $F\to F/F_\D$ soit nul, autrement dit $F_\D=F$, c'est-à-dire que $\D$ définit un support fini de $F$.

On montre alors l'équivalence entre (\ref{itpsc1}) et (\ref{itpsc2}) par récurrence sur $n$ en utilisant les propositions~\ref{pr-psfn-ind} et~\ref{pr-psfn-sec} et l'observation formelle que le noyau d'un épimorphisme entre foncteurs $psf'_n$ est $psf'_{n-1}$.
\end{proof}

\begin{cor}\label{cor-pspfn} La propriété $pf_n$ entraîne la propriété $psf_n$.
\end{cor}

La réciproque est évidemment fausse : si $\C$ n'a qu'un nombre fini d'objets, alors tout foncteur de $\fct(\C,\E)$ possède la propriété $psf_\infty$, mais une somme directe d'un nombre infini de copies d'un objet non nul de $\fct(\C,\E)$ ne vérifie jamais $pf_0$.

\begin{cor}\label{cor-pfpsf} Supposons que, pour tous objets $a$ et $b$ de $\C$, l'ensemble $\C(a,b)$ est fini. Étant donnés un foncteur $F$ de $\fct(\C,\E)$ et $n\in\mathbb{N}\cup\{\infty\}$, les assertions suivantes sont équivalentes :
\begin{enumerate}
\item $F$ appartient à $\pf_n(\fct(\C,\E))$ ;
\item $F$ appartient à $\ppsf_n(\fct(\C,\E))$ et est à valeurs dans $\pf_n(\E)$.
\end{enumerate}
\end{cor}

\begin{proof} Supposons que $F$ est $pf_n$. Alors $F$ est $psf_n$ par le corollaire~\ref{cor-pspfn}. Si $t$ est un objet de $\C$, le foncteur $\fct(\C,\E)\to\E$ d'évaluation en $t$ est exact et est adjoint à gauche au foncteur $X\mapsto X^{\C(-,t)}$ (de façon duale de l'isomorphisme \eqref{eq-yon}, page~\pageref{eq-yon}), qui est exact et commute aux colimites parce que $\C(-,t)$ prend ses valeurs dans les ensembles finis. La proposition~\ref{pr-pfnad} montre donc que $F(t)$ appartient à $\pf_n(\E)$.

La réciproque découle de l'implication (a)$\Rightarrow$(c) de la proposition~\ref{pr-psfn-ind} combinée aux propositions~\ref{yoneda-pfn} et~\ref{pr-pres-concr}.
\end{proof}

\subsection{Supports de présentation supérieure des bifoncteurs}

Dans ce paragraphe, $\D'$ désigne une sous-catégorie pleine de $\C'$.

\begin{pr}\label{pr-nsuprod} Soient $B$ un foncteur de $\fct(\C'\times\C,\E)$ et $n\in\mathbb{N}\cup\{\infty\}$. Les assertions suivantes sont équivalentes :
\begin{enumerate}
\item\label{itt1} $\D'\times\D$ est un support de $n$-présentation de $B$ ;
\item\label{itt2} pour tout objet $t$ de $\C$, $\D'$ est un support de $n$-présentation du foncteur $B(-,t)$ de $\fct(\C',\E)$, et pour tout objet  $t'$ de $\C'$, $\D$ est un support de $n$-présentation du foncteur $B(t',-)$ de $\fct(\C,\E)$.
\end{enumerate}
\end{pr}

\begin{proof} L'implication (\ref{itt1})$\Rightarrow$(\ref{itt2}) découle de l'équivalence (a)$\Leftrightarrow$(b) de la proposition~\ref{pr-psfn-ind}. La réciproque se déduit de l'équivalence (a)$\Leftrightarrow$(c) de la proposition~\ref{pr-psfn-ind} et d'un argument standard d'objet bisimplicial \cite[§\,8.5]{Wei}.
\end{proof}

 Rappelons que \emph{l'homologie de Hochschild} (parfois appelée simplement homologie et notée $H_*(\C;B)$, comme dans \cite[§\,C.10]{Lod}, qui explicite le complexe standard la calculant) de la catégorie $\C$ à coefficients dans un bifoncteur $B$ de $\fct(\C^\op\times\C,\E)$, notée $HH_*(\C;B)$, est par définition donnée par les foncteurs dérivés à gauche, évalués sur $B$, du foncteur de cofin \cite[chapitre~IX, §\,6]{ML-cat} sur $\C$.
 
 Le résultat classique suivant découle par exemple de \cite[(C.10.1)]{Lod}.
 
\begin{lm}\label{lmhh} Soient $F$ un foncteur de $\fct(\C,\E)$ et $t$ un objet de $\C$. Notons $B$ le bifoncteur de $\fct(\C^\op\times\C,\E)$ donné par $B(x,y)=F(y)[\C(x,t)]$. Alors $HH_0(\C;B)\simeq F(t)$ naturellement en $F$ et $t$, et $HH_i(\C;B)=0$ pour $i>0$.
\end{lm}

\begin{pr}\label{pr-supp-HH} Soient $B$ un foncteur de $\fct(\C^\op\times\C,\E)$ et  $n\in\mathbb{N}\cup\{\infty\}$. Supposons que $(\C^\op\times\D)\cup (\D^\op\times\C)$ est un support de $n$-présentation de $B$. Alors le morphisme
$$HH_i(\D;\iota^*B)\to HH_i(\C;B)$$
induit par le foncteur d'inclusion $\iota : \D\to\C$ est un isomorphisme pour $i<n$ et un épimorphisme pour $i<n+1$.
\end{pr}

\begin{proof} Par la proposition~\ref{pr-psfn-ind}, le bifoncteur $B$ admet une $n$-présentation par des sommes directes de foncteurs de la forme $M[\C(-_1,x)\times\C(y,-_2)]$, où $M$ est un objet de $\E$ et $x$ et $y$ sont des objets de $\C$ dont l'un appartient à $\D$. La conclusion s'obtient alors par un argument classique de suite spectrale à partir du lemme~\ref{lmhh}.
\end{proof}

\section{Propriété $pf_n$ et produits tensoriels}\label{s-pfpt}

Pour obtenir des informations sur le caractère $pf_\infty$ des foncteurs, nous aurons besoin de résultats sur le comportement des propriétés $pf_n$ vis-à-vis des produits tensoriels. Ces résultats font l'objet de la présente section ; ils seront utilisés notamment dans les démonstrations des théorèmes~\ref{th-pol-pfn} et~\ref{th-pfigl}.

\subsection{Produits tensoriels}

Le cadre général pour les produits tensoriels que nous considérons est le suivant. Si l'on suppose que la catégorie de Grothendieck $\E$ est $k$-linéaire, on dispose d'un produit tensoriel (au-dessus de $k$) $$\otimes : (k\Md)\times\E\to\E\;.$$ 
Ce produit tensoriel est le bifoncteur qui commute aux colimites par rapport à chaque variable et tel que $k\otimes X=X$. En l'appliquant au but, on en déduit un \emph{produit tensoriel extérieur}: 
$$\boxtimes : \F(\C;k)\times\E\to\fct(\C,\E)$$ 
tel que $(F\boxtimes M)(t)=F(t)\otimes M$.

Un cas fondamental est celui où $\E=\F(\D,k)$ : utilisant l'isomorphisme canonique de catégories $\fct(\C,\F(\D;k))\simeq\F(\C\times\D;k)$, le produit tensoriel précédent fournit un bifoncteur, toujours appelé produit tensoriel extérieur et noté $\boxtimes : \F(\C;k)\times\F(\D;k)\to\F(\C\times\D;k)$ donné par $(F\boxtimes G)(t,u)=F(t)\otimes G(u)$. Lorsque $\C=\D$, la précomposition par la diagonale $\C\to\C\times\C$ de $F\boxtimes G$ égale le produit tensoriel usuel $F\otimes G$, donné par $(F\otimes G)(t)=F(t)\otimes G(t)$.

%Par exemple, si $c$ et $d$ sont des objets de $\C$ et $\D$ respectivement, on dispose d'un isomorphisme canonique $P^c_\C\boxtimes P^d_\D\simeq P^{(c,d)}_{\C\times\D}$.

La formule de K\"unneth suivante, classique et importante, ne sera pas utilisée dans la suite de l'article ; nous la donnons à titre de motivation.

\begin{pr}\label{pr-kunneth} Supposons que $k$ est un corps et que $\E$ est $k$-linéaire. Soient $F$, $G$ des objets de $\F(\C;k)$, $U$ et $V$ des objets de $\E$ et $n\in\mathbb{N}\cup\{\infty\}$. Le morphisme naturel d'espaces vectoriels gradués
\begin{equation}\label{eq-mpd}
\mathrm{Ext}^*_{\F(\C;k)}(F,G)\otimes\mathrm{Ext}^*_\E(U,V)\to\mathrm{Ext}^*_{\fct(\C,\E)}(F\boxtimes U,G\boxtimes V)
\end{equation}
est bijectif en degrés $<n$ et injectif en degrés $<n+1$ dans chacun des deux cas suivants :
\begin{enumerate}
\item $F$ et $U$ possèdent la propriété $pf_{n-1}$ ;
\item $F$ appartient à $\pf_{n-1}(\F(\C;k))$ et $G$ est à valeurs de dimensions finies.
\end{enumerate}
\end{pr}

\begin{proof}[Esquisse de démonstration]
En résolvant $F$, on se ramène au cas où ce foncteur est un projectif de type fini $P^t_\C=k[\C(t,-)]$. Utilisant l'isomorphisme \eqref{eq-yon} (page~\pageref{eq-yon}), on voit que le morphisme \eqref{eq-mpd} s'identifie alors au morphisme canonique $G(t)\otimes\mathrm{Ext}^*_\E(U,V)\to\mathrm{Ext}^*_\E(U,G(t)\otimes V)$, ce qui permet de conclure.
\end{proof}

\begin{rem}
Le cas $\E=\F(\D;k)$ de la proposition \ref{pr-kunneth} prend une forme plus symétrique ; il est explicitement énoncé et démontré\,\footnote{La démonstration donnée dans un cas particulier fonctionne pareillement pour toutes les catégories de foncteurs du type mentionné.} par Franjou \cite[proposition~1.4.1]{Fran96}. Cette formule de K\"unneth joue un rôle clé dans les calculs de \cite{Fran96} et de \cite{FFSS}. L'importance de cette formule de K\"unneth dans un cadre général est également soulignée dans \cite{DT}.
\end{rem}

\subsection{Changement de catégorie but}\label{ssect-ccb}

Nous commençons par le changement de catégorie but de $k\Md$ à $K\Md$ lorsque $K$ est une $k$-algèbre.

La propriété suivante s'applique notamment lorsque $k\to K$ est une extension de corps. Elle sera utile lorsque nous emploierons des résultats de \cite{DTV} nécessitant un corps de coefficients assez gros au but des catégories de foncteurs.

\begin{pr}\label{pr-chbpfn} Soient $F$ un foncteur de $\F(\C;k)$ et $n\in\mathbb{N}\cup\{\infty\}$. Si $K$ est une $k$-algèbre plate (resp. fidèlement plate), alors $F\otimes K$ appartient à $\pf_n(\F(\C;K))$ si (resp. si et seulement si) $F$ appartient à $\pf_n(\F(\C;k))$.
\end{pr}

\begin{proof} Comme $K$ est $k$-plate, le foncteur d'extension des scalaires au but $\Phi:= -\otimes K : \F(\C;k)\to\F(\C;K)$ est exact ; de plus, son adjoint à droite, la post-composition par la restriction des scalaires, est exact et commute aux colimites. Par la proposition~\ref{pr-pfnad}, il s'ensuit que $\Phi$ préserve la propriété $pf_n$. Si de plus $K$ est {\em fidèlement} plate, $\Phi$ est également fidèle, de sorte que la réciproque découle des propositions~\ref{pr-pfn-efonc} et~\ref{rq-prestf}.
\end{proof}

La proposition~\ref{pr-chbpfn} possède un analogue pour les supports de présentation supérieure :

\begin{pr}\label{pr-chbpsfn} Soient $F$ un foncteur de $\F(\C;k)$, $\D$ une sous-catégorie pleine de $\C$ et $n\in\mathbb{N}\cup\{\infty\}$. Si $K$ est une $k$-algèbre $K$ plate (resp. fidèlement plate), alors $\D$ est un support de $n$-présentation de $F\otimes K$ dans $\F(\C;K)$ si (resp. si et seulement si) $\D$ est un support de $n$-présentation de $F$ dans $\F(\C;k)$.
\end{pr}

\begin{proof} Cela résulte de l'équivalence (a)$\Leftrightarrow$(c) de la proposition~\ref{pr-psfn-ind}, puisqu'une suite de $k$-modules est exacte seulement si (resp. si et seulement si) la suite de $K$-modules obtenue en appliquant le foncteur $- \otimes K$ l'est.
\end{proof}

\begin{cor}\label{cor-chbpsfn} Soient $F$ un foncteur de $\F(\C;k)$ et $n\in\mathbb{N}\cup\{\infty\}$. Si $K$ est une $k$-algèbre fidèlement plate, alors $F$ vérifie la propriété $psf_n$ dans $\F(\C;k)$ si et seulement si $F\otimes K$ vérifie la propriété $psf_n$ dans $\F(\C;K)$.
\end{cor}

Les énoncés suivants, qui nous serviront en fin d'article, concernent le changement de catégorie but plus général de $k\Md$ à $\E$, lorsque $\E$ est une catégorie de Grothendieck $k$-linéaire. Ils nécessitent une hypothèse de finitude forte sur $\F(\C;k)$. Les notions de foncteur {\em absolument simple}, d'objet {\em fini} (nommé aussi {\em de longueur finie}), de catégorie {\em localement finie} ou {\em localement de type fini} qui y interviennent sont rappelées par exemple dans \cite[définition~3.3]{DTV} et \cite[pages~92, 368 et 371]{Pop}.

\begin{lm}\label{lm-pte-fini} Supposons que $k$ est un corps et que la catégorie $\E$ est $k$-linéaire. Supposons également que la catégorie $\F(\C;k)$ est localement finie et que tous ses objets simples sont absolument simples. Supposons enfin que $\E$ est localement de type fini, ou que les foncteurs simples de $\F(\C;k)$ sont à valeurs de dimensions finies. Alors tout foncteur à support fini de $\fct(\C,\E)$ possède une filtration finie dont les sous-quotients sont isomorphes à des produits tensoriels extérieurs de la forme $S\boxtimes X$, où $S$ est un foncteur simple de $\F(\C;k)$ et $X$ un objet de $\E$.
\end{lm}

\begin{proof} Soit $F$ un foncteur à support fini de $\fct(\C,\E)$ : il existe un nombre fini d'objets $t_1,\dots,t_n$ de $\C$ et des objets $X_1,\dots,X_n$ de $\E$ tels que $F$ soit un quotient de $\bigoplus_{i=1}^n k[\C(t_i,-)]\boxtimes X_i$, par la proposition~\ref{pr-psfn-ind}. Comme $\F(\C;k)$ est localement finie, les foncteurs projectifs de type fini $k[\C(t_i,-)]$ sont finis : ils possèdent une filtration finie dont les sous-quotients sont simples, donc absolument simples. L'exactitude en chaque variable de $\boxtimes$ ($k$ est un corps) et le lemme~\ref{lm-abss} ci-dessous permettent donc de conclure.
\end{proof}

\begin{lm}\label{lm-abss} Supposons que $k$ est un corps et que la catégorie $\E$ est $k$-linéaire. Soit $S$ un objet absolument simple de $\F(\C;k)$. Supposons de plus que $S$ est à valeurs de dimensions finies, ou bien que $\E$ est localement de type fini. Alors l'image essentielle du foncteur $S\boxtimes - : \E\to\fct(\C,\E)$ est stable par sous-quotient.
\end{lm}

\begin{proof} Le foncteur $S\boxtimes - $ est exact, pleinement fidèle, et commute aux colimites. Comme la catégorie $\fct(\C,\E)$ est engendrée par les foncteurs $T[\C(t,-)]$ (où $t$ est un objet de $\C$ et $T$ un objet de $\E$, qu'on peut supposer de type fini si $\E$ est localement de type fini), il suffit donc de montrer que tout sous-foncteur $F$ de $S\boxtimes X$ image d'un morphisme $f : T[\C(t,-)]\to S\boxtimes X$ est de la forme $S\boxtimes Y$, où $Y$ est un sous-objet de $X$ (dans $\E$). Utilisons les isomorphismes
$$\fct(\C,\E)(T[\C(t,-)],S\boxtimes X)\simeq\E(T,S(t)\otimes X)\simeq S(t)\otimes\E(T,X)$$
où le premier isomorphisme est un cas particulier de \eqref{eq-yon} (page~\pageref{eq-yon}) et le second est valide si $S$ est à valeurs de dimensions finies ou si $T$ est de type fini. Écrivons l'image de $f$ dans $S(t)\otimes\E(T,X)$ sous la forme $\sum_{i=1}^n\varphi_i\otimes\psi_i$, où $(\varphi_i)_i$ est une famille {\em libre} de l'espace vectoriel $S(t)$ : comme $S$ est absolument simple, le morphisme $P^t_\C\to S^{\oplus n}$ dont les composantes dans $\F(\C;k)(P^t_\C,S)\simeq S(t)$ sont les $\varphi_i$ est {\em surjectif}. Par conséquent, l'image $F$ de $f$ égale celle du morphisme
$$S^{\oplus n}\boxtimes T\simeq S\boxtimes T^{\oplus n}\xrightarrow{S\boxtimes\Psi}S\boxtimes X$$
c'est-à-dire $S\boxtimes\mathrm{Im}\Psi$, où $\Psi : T^{\oplus n}\to X$ est le morphisme dont les composantes $T\to X$ sont les $\psi_i$, d'où la conclusion recherchée avec $Y=\mathrm{Im}\Psi$.
\end{proof}

L'énoncé suivant constitue, au-delà du corollaire~\ref{cor-pspfn}, l'un de nos seuls critères généraux pour obtenir la propriété $psf_\infty$. Il sera utilisé, par l'intermédiaire du corollaire~\ref{cor-pfinf-lfcor}, pour établir le théorème~\ref{th-pfigl}.

\begin{pr}\label{pr-psfinf-lfcor} Supposons que $k$ est un corps, que la catégorie $\E$ est $k$-linéaire, que pour tous objets $x$ et $y$ de $\C$, l'ensemble $\C(x,y)$ est fini, et que la catégorie $\F(\C;k)$ est localement finie. Alors tout foncteur à support fini de $\fct(\C,\E)$ possède la propriété $psf_\infty$. 
\end{pr}

\begin{proof} Grâce au corollaire~\ref{cor-chbpsfn}, par extension des scalaires au but, on peut supposer que $k$ est algébriquement clos. Comme les ensembles de morphismes entre deux objets de $\C$ sont supposés finis, les simples de $\F(\C;k)$ prennent des valeurs de dimensions finies, et ils sont absolument simples (lemme de Schur). On peut donc utiliser le lemme~\ref{lm-pte-fini} qui (combiné à la proposition~\ref{pr-pfn-sec}) montre qu'il suffit de voir que les foncteurs $S\boxtimes X$, où $S$ est un foncteur simple de $\F(\C;k)$ et $X$ un objet de $\E$, vérifient $psf_\infty$ dans $\fct(\C,\E)$. En effet, $X$ est $pf_\infty$, donc $psf_\infty$, dans $\F(\C;k)$, car cette catégorie est localement noethérienne, par la proposition~\ref{anc-rqc}. La proposition~\ref{pr-psfn-ind} montre que la propriété $psf_\infty$ est préservée par $-\boxtimes X$ (puisque $k$ est un corps), d'où la conclusion.
\end{proof}

\begin{cor}\label{cor-pfinf-lfcor} Sous les hypothèses de la proposition~\ref{pr-psfinf-lfcor}, tout foncteur de $\fct(\C,\E)$ à support fini et à valeurs dans $\pf_\infty(\E)$ appartient à $\pf_\infty(\fct(\C,\E))$.
\end{cor}

\begin{proof} Cela découle de la proposition~\ref{pr-psfinf-lfcor} et du corollaire~\ref{cor-pfpsf}.
\end{proof}

\subsection{Produits tensoriels de foncteurs $pf_n$ : cas plat}\label{ssect-ptpfn-plat}

\begin{pr}\label{prpt1} Soient $n\in\mathbb{N}\cup\{\infty\}$, $F$ et $G$ des foncteurs de $\pf_n(\F(\C;k))$ et $\pf_n(\F(\D;k))$ respectivement. On suppose que $F$ ou $G$ prend ses valeurs dans les $k$-modules plats. Alors $F\boxtimes G$ appartient à $\pf_n(\F(\C\times\D;k))$.
\end{pr}

\begin{proof} Le résultat se déduit de la proposition~\ref{pr-pres-concr} et des isomorphismes canoniques $P_\C^c\boxtimes P_\D^d\simeq P_{\C\times\D}^{(c,d)}$.
\end{proof}

\begin{pr}\label{pr-ptens-int} Supposons que la catégorie $\C$ possède des coproduits finis. Soit $n\in\mathbb{N}\cup\{\infty\}$. Alors le produit tensoriel de deux foncteurs de $\pf_n(\F(\C;k))$ dont l'un est à valeurs $k$-plates appartient à $\pf_n(\F(\C;k))$.
\end{pr}

\begin{proof} Cela découle de la proposition~\ref{pr-pres-concr} et des isomorphismes canoniques $P_\C^c\otimes P_\C^d\simeq P_\C^{c*d}$, où $*$ désigne le coproduit de $\C$.
\end{proof}

Les propositions~\ref{prpt1} et~\ref{pr-ptens-int} possèdent les réciproques suivantes.

\begin{pr}\label{prptrec1} Soient $n\in\mathbb{N}\cup\{\infty\}$, $F$ et $G$ des foncteurs de $\F(\C;k)$ et $\F(\D;k)$ respectivement. On suppose que toutes les valeurs de $F$ et de $G$ sont $k$-plates, et que chacun d'entre eux possède au moins une valeur fidèlement plate. Alors, si $F\boxtimes G$ appartient à $\pf_n(\F(\C\times\D;k))$, $F$ (resp. $G$) est dans $\pf_n(\F(\C;k))$ (resp. $\pf_n(\F(\D;k))$).
\end{pr}

\begin{proof} Comme $F$ est à valeurs plates, le foncteur $\Phi_F:=F\boxtimes - :\F(\D;k)\to\F(\C\times\D;k)$ est exact. Comme $F$ prend une valeur fidèlement plate, $\Phi_F$ est également fidèle. De plus, $\Phi_F$ commute aux colimites. Il s'ensuit que, si $F\boxtimes G$ est de type fini, alors $G$ est également de type fini, et $F$ également. Cela établit le cas $n=0$. On raisonne ensuite par récurrence, supposant $0<n<\infty$ et l'assertion vraie pour la propriété $pf_{n-1}$.

Supposons donc que $F\boxtimes G$ vérifie $pf_n$ : l'hypothèse de récurrence montre que $F$ et $G$ sont $pf_{n-1}$. Par conséquent, la proposition~\ref{prpt1} montre que $\Phi_F$ préserve les objets $pf_i$ pour $i\le n-1$. La proposition~\ref{pr-pfn-efonc} montre alors (compte-tenu de la première partie de la démonstration) que la propriété $pf_n$ pour $\Phi_F(G)$ implique la même propriété pour $G$, d'où notre assertion.
\end{proof}

\begin{pr}\label{prptrec2} Supposons que $\C$ possède des coproduits finis. Soient $F$ et $G$ des foncteurs de $\F(\C;k)$ et $n\in\mathbb{N}\cup\{\infty\}$. On suppose que $F\otimes G$ appartient à $\pf_n(\F(\C;k))$, que $F$ et $G$ sont à valeurs plates, et que chacun d'entre eux prend une valeur fidèlement plate. Alors $F$ et $G$ appartiennent à $\pf_n(\F(\C;k))$.
\end{pr}

\begin{proof} Cette proposition s'établit comme la précédente, en utilisant la proposition~\ref{pr-ptens-int} au lieu de la proposition~\ref{prpt1}.
\end{proof}

\subsection{Cas d'un anneau de base principal}\label{ssect-ptpfn-princ}

\begin{hyp} Dans toute la section~\ref{ssect-ptpfn-princ}, on suppose que l'anneau $k$ est principal.
\end{hyp}

Soient $F$ et $G$ des foncteurs de $\F(\C;k)$ et $\F(\D;k)$ respectivement. On notera $\mathbf{Tor}(F,G)$ le foncteur de $\F(\C\times\D;k)$ composé de $\C\times\D\xrightarrow{F\times G}k\Md\times k\Md\xrightarrow{\Tor^k_1}k\Md$. Lorsque $\C=\D$, on notera $\Tor(F,G)$ le foncteur de $\F(\C;k)$ composé de la diagonale $\C\to\C\times\C$ et de $\mathbf{Tor}(F,G)$.

La proposition~\ref{prpt1} possède la généralisation suivante.

\begin{pr}\label{pr-ptpfn-tor} Soient $F$ (resp. $G$) un foncteur de $\F(\C;k)$ (resp. $\F(\D;k)$) et $n\in\mathbb{N}\cup\{\infty\}$. On suppose que $F$ et $G$ possèdent la propriété $pf_n$. Alors le foncteur $F\boxtimes G$ appartient à $\pf_n(\F(\C\times\D;k))$ si et seulement si $\mathbf{Tor}(F,G)$ appartient à $\pf_{n-2}(\F(\C\times\D;k))$.
\end{pr}

\begin{proof} Soient $P_\bullet$ et $Q_\bullet$ des résolutions projectives de $F$ et $G$ respectivement, de type fini jusqu'en degré $n$. Le complexe total de $P_\bullet\boxtimes Q_\bullet$ est constitué de foncteurs projectifs de type fini de $\F(\C\times\D;k)$ ; comme $P_\bullet$ et $Q_\bullet$ prennent leurs valeurs dans les $k$-modules projectifs, l'homologie de ce complexe est $F\boxtimes G$ en degré $0$, $\mathbf{Tor}(F,G)$ en degré $1$ et $0$ ailleurs. Le corollaire~\ref{cor-pfn-compl} (et l'exemple~\ref{expfinf}(b)) montre alors que $F\boxtimes G$ est $pf_n$ dès lors que $\mathbf{Tor}(F,G)$ vérifie $pf_{n-2}$. La réciproque s'établit de façon analogue, à partir de la proposition~\ref{pr-pfn-sec}.
\end{proof}

On démontre de la même façon : 

\begin{pr}\label{pr-ptpfn-torint} Supposons que $\C$ possède des coproduits finis. Soient $n\in\mathbb{N}\cup\{\infty\}$, $F$ et $G$ des foncteurs de $\pf_n(\F(\C;k))$. Le foncteur $F\otimes G$ appartient à $\pf_n(\F(\C;k))$ si et seulement si $\Tor(F,G)$ est dans $\pf_{n-2}(\F(\C;k))$.
\end{pr}

On peut facilement, à partir de là, exhiber un produit tensoriel non $pf_2$ de foncteurs sur une catégorie additive vérifiant chacun la propriété $pf_2$, comme l'illustre l'exemple suivant.

\begin{ex}\label{ex-prodtens-pasplat} Supposons que $\C$ est la catégorie des groupes abéliens de type fini et que $k=\mathbb{Z}$. Prenons $F$ et $G$ égaux au foncteur d'inclusion vers $\mathbf{Ab}$ et considérons, pour un nombre premier $p$, le foncteur $T_p$ de $\F(\C;\mathbb{Z})$ associant à un groupe abélien de type fini sa composante $p$-primaire. Le foncteur $F$ est représentable, donc projectif de type fini, et a fortiori $pf_\infty$ dans $\add(\C,\mathbf{Ab})$. Le théorème~\ref{thm-add-Z} ci-après implique qu'il appartient également à $\pf_\infty(\F(\C;\mathbb{Z}))$. En revanche, $\Tor(F,F)$ \emph{n'}est \emph{pas} de type fini, car il est isomorphe à la somme directe infinie sur les nombres premiers $p$ des $\Tor(T_p,T_p)$, qui sont tous non nuls.
\end{ex}

Nous allons maintenant donner des critères garantissant la propriété $pf_n$ pour un produit tensoriel de foncteurs à valeurs dans $k\Md$ sans rien supposer a priori des groupes de torsion. Nous aurons besoin d'hypothèses supplémentaires, qui font intervenir la torsion des valeurs de chaque foncteur. Nous introduirons à cet effet des conditions notées $\mathbf{Tpf}_n$ (voir la notation~\ref{notatpf}), qui impliquent $pf_n$, sont stables par produit tensoriel et sont faciles à manier.

\medskip

Si $V$ est un $k$-module, nous noterons $V_\tor$ son sous-module de torsion. Pour $x\in k$, nous désignerons par ${}_x V$ le sous-module $\Hom_k(k/(x),V)$ des éléments de $V$ annulés par $x$. Pour un foncteur $F$ de $\F(\C;k)$, nous noterons $F_\tor$ (resp. ${}_x F$) le sous-foncteur de $F$ obtenu par postcomposition de $F$ avec $V\mapsto V_\tor$ (resp. $V\mapsto{}_x V$). Nous dirons que $F$ est {\em de torsion bornée} s'il existe $x\in k\setminus\{0\}$ tel que l'inclusion ${}_x F\hookrightarrow F_\tor$ soit une égalité --- on dit alors que $F$ est de torsion {\em bornée par $x$}.

On rappelle par ailleurs que, si $\A$, $\B$ et $\E$ sont des catégories $k$-linéaires, un foncteur $T : \A\times\B\to\E$ est dit {\em $k$-bilinéaire} si, pour tous objets $a, a'$ de $\A$ et $b, b'$ de $\B$, l'application $\A(a,a')\times\B(b,b')\to\E(T(a,b),T(a',b'))$ induite par $T$ est $k$-bilinéaire. On notera $\mathbf{Bil}(\A\times\B;k)$ la sous-catégorie pleine de $\F(\A\times\B;k)$ constituée des foncteurs $k$-bilinéaires.

Dans le lemme suivant, on note $k/(x)\md$ la sous-catégorie pleine des modules de type fini de $k/(x)\Md$, qu'on peut voir elle-même comme sous-catégorie pleine de $k\Md$. 
%si $x$ est un élément non nul de $k$, on note $\overline{\A}_x$ la sous-catégorie pleine de $k\Md$ constituée des modules dont la torsion est bornée par $x$, et $\A_x$ la sous-catégorie pleine des modules de type fini de $\overline{\A}_x$.

\begin{lm}\label{lm-bil-ln} Soient $x$ un élément non nul de $k$ et $B$ un foncteur de $\mathbf{Bil}(k/(x)\md\times k/(x)\md;k)$. Les assertions suivantes sont équivalentes :
\begin{enumerate}
\item\label{itnp1} $B$ prend ses valeurs dans les $k$-modules de type fini ;
\item\label{itnp2} $B$ est noethérien ;
\item\label{itnp3} $B$ appartient à $\pf_\infty(\mathbf{Bil}(k/(x)\md\times k/(x)\md;k))$.
\end{enumerate}
\end{lm}

\begin{proof} Comme l'anneau $k$ est principal, tout objet de $k/(x)\md$ est isomorphe à une somme directe finie de $k$-modules indécomposables, et l'ensemble des classes d'isomorphisme d'objets indécomposables, dont nous noterons $\s$ un ensemble de représentants, est {\em fini}.

Comme le foncteur 
$$ev_\s : \mathbf{Bil}(k/(x)\md\times k/(x)\md;k)\to(k\Md)^{\s\times\s}\qquad X\mapsto (X(U,V))_{(U,V)\in\s\times\s}$$
est exact et pleinement fidèle, tout objet $X$ de $\mathbf{Bil}(k/(x)\md\times k/(x)\md;k)$ tel que $ev_\s(X)$ soit noethérien est lui-même noethérien. Comme $\s$ est fini et que $k$ est un anneau noethérien, cela montre l'implication (\ref{itnp1})$\Rightarrow$(\ref{itnp2}).

La catégorie $\mathbf{Bil}(k/(x)\md\times k/(x)\md;k)$ est engendrée par les foncteurs $k/(x)\md(U,-)\boxtimes k/(x)\md(V,-)$ pour $U, V$ dans $k/(x)\md$. Comme ces foncteurs sont à valeurs de type fini, on en déduit l'implication réciproque (\ref{itnp2})$\Rightarrow$(\ref{itnp1}) et le caractère localement noethérien de la catégorie $\mathbf{Bil}(k/(x)\md\times k/(x)\md;k)$. Celui-ci entraîne l'équivalence de (\ref{itnp2}) et (\ref{itnp3}), grâce à la proposition~\ref{anc-rqc}.
\end{proof}

\begin{lm}\label{lm-res-torB} Soient $x$ un élément non nul de $k$ et $T : k/(x)\Md\times k/(x)\Md\to k\Md$ un foncteur $k$-bilinéaire vérifiant les deux propriétés suivantes :
\begin{enumerate}
\item si $U$ et $V$ sont des objets de $k/(x)\md$, $T(U,V)$ est un $k$-module de type fini ;
\item $T$ commute aux colimites filtrantes.
\end{enumerate}

Alors il existe une résolution de $T$ par des sommes directes finies de foncteurs du type $(U,V)\mapsto ({}_a U)\otimes ({}_b V)$, où $(a,b)\in k^2$.
\end{lm}

\begin{proof} L'équivalence (\ref{itnp1})$\Leftrightarrow$(\ref{itnp3}) du lemme~\ref{lm-bil-ln} et la proposition~\ref{pr-pres-concr} montrent que la restriction de $T$ à $k/(x)\md\times k/(x)\md$ possède une résolution par des sommes directes finies de foncteurs du type $(U,V)\mapsto ({}_a U)\otimes ({}_b V)$. \'Etant donné que ces foncteurs, ainsi que $T$, commutent aux colimites filtrantes, cette résolution s'étend en une résolution de $T$ de la forme souhaitée.
\end{proof}

\begin{pr}\label{pr-ptpf-tb} Soient $F$ (resp. $G$) un foncteur de $\F(\C;k)$ (resp. $\F(\D;k)$) et $n, m\in\mathbb{N}\cup\{\infty\}$. On suppose que :
\begin{enumerate}
\item $F$ et $G$ possèdent la propriété $pf_n$ ;
\item $F$ et $G$ sont de torsion bornée ;
\item pour tout $x\in k\setminus\{0\}$, ${}_x F$ et ${}_x G$ vérifient la propriété $pf_m$.
\end{enumerate}

Alors :
\begin{enumerate}
\item[$\mathrm{(A)}$] le foncteur $F\boxtimes G$ appartient à $\pf_{\min(n,m+2)}(\F(\C\times\D;k))$ ;
\item[$\mathrm{(B)}$] le foncteur $\mathbf{Tor}(F,G)$ appartient à $\pf_{\min(n,m)}(\F(\C\times\D;k))$ ;
\item[$\mathrm{(C)}$] si $T : (k\Md)\times (k\Md)\to k\Md$ est un bifoncteur $k$-bilinéaire qui commute aux colimites filtrantes et envoie toute paire de $k$-modules de type fini sur un $k$-module de type fini, alors le foncteur  $T\circ (F\times G)$ appartient à $\pf_{\min(n,m)}(\F(\C\times\D;k))$.
\end{enumerate}
\end{pr}

\begin{proof} Soit $x$ un élément non nul de $k$. On va montrer, plus précisément, les propriétés suivantes (où $n\in\mathbb{Z}\cup\{\infty\}$) :
\begin{enumerate}
\item[$\mathrm{(A)}_n$] si $F$ (resp. $G$) est un foncteur de $\pf_n(\F(\C;k))$ (resp. $\pf_n(\F(\D;k))$), à valeurs dans $k/(x)\Md$, et que pour tout $x\in k\setminus\{0\}$, ${}_x F$ et ${}_x G$ vérifient la propriété $pf_{n-2}$, alors $F\boxtimes G$ appartient à $\pf_n(\F(\C\times\D;k))$ ;
\item[$\mathrm{(B)}_n$] si $F$ (resp. $G$) est un foncteur de $\pf_n(\F(\C;k))$ (resp. $\pf_n(\F(\D;k))$), à valeurs dans $k/(x)\Md$, et que pour tout $x\in k\setminus\{0\}$, ${}_x F$ et ${}_x G$ vérifient la propriété $pf_n$, alors le foncteur $\mathbf{Tor}(F,G)$ appartient à $\pf_n(\F(\C\times\D;k))$ ;
\item[$\mathrm{(C)}_n$] si $F$ (resp. $G$) est un foncteur de $\pf_n(\F(\C;k))$ (resp. $\pf_n(\F(\D;k))$), à valeurs dans $k/(x)\Md$, que pour tout $x\in k\setminus\{0\}$, ${}_x F$ et ${}_x G$ vérifient la propriété $pf_n$, et que $T : (k\Md)\times (k\Md)\to k\Md$ est comme dans (C), alors le foncteur  $T\circ (F\times G)$ appartient à $\pf_n(\F(\C\times\D;k))$ ;
\item[$\mathrm{(D)}_n$] si $F$ (resp. $G$) est un foncteur de $\pf_n(\F(\C;k))$ (resp. $\pf_n(\F(\D;k))$), à valeurs dans $k/(x)\Md$, et si pour tout $x\in k\setminus\{0\}$, ${}_x F$ et ${}_x G$ vérifient la propriété $pf_n$, alors pour tout $(a,b)\in k^2$, le foncteur $({}_a F)\boxtimes ({}_b G)$ vérifie $pf_n$.
\end{enumerate}

On procède selon le schéma de récurrence suivant :
$$\mathrm{(A)}_n\Rightarrow\mathrm{(D)}_n\Rightarrow\mathrm{(C)}_n\Rightarrow\mathrm{(B)}_n\Rightarrow\mathrm{(A)}_{n+2}$$
(Comme ces propriétés sont vides pour $n<0$, il n'y a pas d'initialisation).

Si $F$ et $G$ vérifient les hypothèses de $\mathrm{(D)}_n$, alors tous leurs sous-foncteurs ${}_a F$ et ${}_b G$ vérifient les hypothèses de $\mathrm{(A)}_n$, d'où l'implication $\mathrm{(A)}_n\Rightarrow\mathrm{(D)}_n$.

L'implication $\mathrm{(D)}_n\Rightarrow\mathrm{(C)}_n$ s'obtient en appliquant le lemme~\ref{lm-res-torB} à la restriction de $T$ à $k/(x)\Md\times k/(x)\Md$, puis le corollaire~\ref{cor-pfn-compl}.

L'implication $\mathrm{(C)}_n\Rightarrow\mathrm{(B)}_n$ résulte de ce que le foncteur $\Tor^k_1$ commute aux colimites filtrantes et envoie toute paire de $k$-modules de type fini sur un module de type fini.

La proposition~\ref{pr-ptpfn-tor} fournit l'implication $\mathrm{(B)}_n\Rightarrow\mathrm{(A)}_{n+2}$.
\end{proof}

\begin{nota}\label{notatpf} On note $\mathbf{Tpf}_n(\C;k)$ la classe des foncteurs de $\F(\C;k)$ vérifiant les propriétés suivantes :
\begin{enumerate}
\item pour tout $x\in k$, le foncteur ${}_x F$ appartient à $\pf_n(\F(\C;k))$ ;
\item la torsion de $F$ est bornée.
\end{enumerate}
\end{nota}

\begin{cor}\label{cor-tpf-ext} Soient $n\in\mathbb{N}\cup\{\infty\}$ et $F$ (resp. $G$) un foncteur de $\mathbf{Tpf}_n(\C;k)$ (resp. $\mathbf{Tpf}_n(\D;k)$). Alors $F\boxtimes G$ et $\mathbf{Tor}(F,G)$ appartiennent à $\mathbf{Tpf}_n(\C\times\D;k)$. 
\end{cor}

\begin{proof} Il est immédiat que les foncteurs $F\boxtimes G$ et $\mathbf{Tor}(F,G)$ sont de torsion bornée, comme $F$ et $G$. Pour vérifier le caractère $pf_n$ de ${}_x (F\boxtimes G)$ et ${}_x \mathbf{Tor}(F,G)$, on applique l'assertion (C) de la proposition~\ref{pr-ptpf-tb} aux bifoncteurs $(U,V)\mapsto {}_x (U\otimes V)$ et $(U,V)\mapsto {}_x \Tor^k_1(U,V)$.
\end{proof}

\begin{cor}\label{cor-tpf-int} Supposons que $\C$ possède des coproduits finis. Soient $n\in\mathbb{N}\cup\{\infty\}$, $F$ et $G$ des foncteurs de $\mathbf{Tpf}_n(\C;k)$. Alors $F\otimes G$ et $\Tor(F,G)$ appartiennent à $\mathbf{Tpf}_n(\C;k)$. 
\end{cor}

\begin{proof} Cela résulte du corollaire~\ref{cor-tpf-ext} et de la proposition~\ref{pr-precompf}, puisque la diagonale de $\C$ est adjointe à droite au coproduit $\C\times\C\to\C$.
\end{proof}

Le corollaire suivant nous servira au §\,\ref{ssct-Z}.

\begin{cor}\label{cor-ptpf-pasplat} Supposons que $\C$ possède des coproduits finis. Soient $n\in\mathbb{N}\cup\{\infty\}$, et $\mathfrak{P}$ une classe d'objets de $\F(\C;k)$ telle que :
\begin{enumerate}
\item $\mathfrak{P}$ est stable par sous-objet ;
\item tout foncteur de $\mathfrak{P}$ est noethérien ;
\item $\mathfrak{P}$ est incluse dans $\pf_n(\F(\C;k))$.
\end{enumerate}

Si $F_1,\dots, F_r$ appartiennent à $\mathfrak{P}$, alors $F_1\otimes\dots\otimes F_r$ appartient à $\pf_n(\F(\C;k))$.
\end{cor}

\begin{proof} Si $F$ est un foncteur noethérien, son sous-foncteur $F_\tor$ est de type fini, donc $F$ est de torsion bornée. Les première et troisième hypothèses sur $\mathfrak{P}$ montrent par ailleurs que ${}_x F$ vérifie $pf_n$ pour tout $x\in k$ et tout $F\in\mathfrak{P}$. Ainsi, $\mathfrak{P}$ est incluse dans $\mathbf{Tpf}_n(\C;k)$, et le corollaire~\ref{cor-tpf-int} permet de conclure.
\end{proof}

%\subsection{La propriété $\mathbf{Tpf}_n$}

\section{La propriété $psf_\infty$ pour les foncteurs polynomiaux sur $\mathbf{P}(A)$}\label{s-psf}

Le {\em lemme de Schwartz} affirme que dans la catégorie $\fct(\mathbf{P}(\mathbb{F}_q);\mathbb{F}_q)$, tout foncteur polynomial possède la propriété $psf_\infty$ \cite[proposition~10.4]{FLS}, ou encore que tout foncteur polynomial à valeurs de dimensions finies est $pf_\infty$ \cite[proposition~10.1]{FLS}. Ces deux énoncés sont équivalents en raison de la finitude des ensembles de morphismes à la source, par le corollaire~\ref{cor-pfpsf}. On peut aussi déduire ces propriétés du caractère localement noethérien de $\fct(\mathbf{P}(\mathbb{F}_q);\mathbb{F}_q)$, démontré par Putman-Sam-Snowden \cite{PSam,SamSn} bien après \cite{FLS}, grâce à la proposition~\ref{anc-rqc}. Toutefois, si l'on remplace le corps fini à la source par un anneau infini, le lien entre les propriétés $pf_\infty$ et $psf_\infty$ est bien moins clair, et la propriété de noethérianité locale de la catégorie de foncteurs tombe en défaut \cite[proposition~11.1]{DTV}.

 Dans cette section, nous reprenons la méthode de démonstration du lemme de Schwartz de Franjou-Lannes-Schwartz \cite[§\,10]{FLS} afin d'en obtenir une version plus générale aussi bien à la source (le corps fini est remplacé par un anneau arbitraire) qu’au but (une catégorie de Grothendieck quelconque plutôt qu’une catégorie d’espaces vectoriels). Nous utilisons ensuite cette généralisation pour améliorer les bornes de stabilité de Betley-Pirashivili \cite{BP} (obtenues par une méthode indépendante) pour l'homologie des monoïdes multiplicatifs de matrices à coefficients polynomiaux.

\begin{nota} Dans toute la section~\ref{s-psf}, $A$ désigne un anneau.
On note $\F(A,k)$ la catégorie de foncteurs $\F(\mathbf{P}(A);k)$.
\end{nota}

\subsection{Résultat principal}\label{ssect-PA}

\begin{defi}
Soient $n\in\mathbb{N}\cup\{\infty\}$ et $m\in\mathbb{N}$. On dit qu'un foncteur $F$ de $\fct(\mathbf{P}(A),\E)$ \emph{vérifie la propriété $\psf_n(m)$} si la sous-catégorie pleine de $\mathbf{P}(A)$ ayant pour unique objet $A^m$ constitue un support de $n$-présentation de $F$.
\end{defi}

Par équivalence de Morita, $F$ vérifie $\psf_n(m)$ si et seulement si la sous-catégorie pleine de $\mathbf{P}(A)$ des facteurs directs de $A^m$ constitue un support de $n$-présentation de $F$. En particulier, si $n$ est fini, un foncteur $F$ vérifie la propriété $psf_n$ si et seulement s'il vérifie $\psf_n(m)$ pour un certain entier $m$. 

Les propriétés $PSF_n(m)$ (au moins pour $n\in\{0,1\}$) sont étroitement reliées à la filtration de la catégorie $\fct(\mathbf{P}(A),\E)$  décrite (dans le cas particulier de $\F(k,k)$, où $k$ est un corps fini) dans \cite[§\,2]{Ku2}.

\medskip

Notons $\widetilde{\Delta}$ le {\em foncteur de décalage} de $\fct(\mathbf{P}(A),\E)$, c'est-à-dire la précomposition par l'endofoncteur $-\oplus A$ de $\mathbf{P}(A)$, et $\Delta$ le {\em foncteur différence}, donné par le scindement canonique $\widetilde{\Delta}=\mathrm{Id}\oplus\Delta$. Alors $\Delta$ est un foncteur exact, qui commute à toutes les limites et colimites. Un foncteur $F$ est {\em polynomial}\,\footnote{Cette définition diffère légèrement de celle d'Eilenberg-MacLane \cite[chapitre~II]{EML} en termes d'effets croisés, qui nous sera plus utile dans la section~\ref{schw2}. Pour des détails sur l'équivalence des deux points de vue, on pourra consulter par exemple \cite[§\,3]{DV-pol}.} de degré au plus $d\in\mathbb{N}$ si et seulement si $\Delta^{d+1}(F)=0$.

La clef de cette section réside dans la proposition suivante, qui constitue une variante de la proposition~\ref{pr-pfn-efonc} pour les propriétés $\psf_n(m)$.

\begin{pr}\label{pdel} Soient $F$ un foncteur de $\fct(\mathbf{P}(A);\E)$ et $n, m\in\mathbb{N}$.
\begin{enumerate}
\item Si $F$ vérifie $\psf_n(m)$, alors il en est de même pour $\Delta(F)$.
\item Si $\Delta(F)$ vérifie $\psf_n(m)$, alors $F$ vérifie $\psf_n(n+m+1)$.
\end{enumerate}
\end{pr}

\begin{proof} La première assertion découle de l'isomorphisme naturel
$$\Delta\big(M[\mathbf{P}(A)(A^i,-)]\big)\simeq \big(M^{\oplus A^i\setminus\{0\}}\big)[\mathbf{P}(A)(A^i,-)],$$
où $M$ est un objet de $\E$ et $i$ un entier naturel, de l'exactitude de $\Delta$ et de la proposition~\ref{pr-psfn-ind}.

Pour établir la suite, on commence par noter que $\Delta$ est adjoint à droite au foncteur $\Phi : \fct(\mathbf{P}(A);\E)\to\fct(\mathbf{P}(A);\E)$ donné par $\Phi(F)(V):=F(V)[V]/F(V)$ ($F(V)$ est le facteur direct naturel de $F(V)[V]$ correspondant à l'élément $0$ de $V$). Les foncteurs $\Delta$ et $\Phi$ sont exacts, et le noyau de $\Delta$ est constitué des foncteurs constants, de sorte qu'on dispose d'un épimorphisme naturel $F(0)\oplus\Phi\Delta(F)\twoheadrightarrow F$ (dont la composante $\Phi\Delta(F)\to F$ est la coünité de l'adjonction). Les foncteurs $\Phi$ et $\Phi\Delta$ envoient un foncteur induit depuis $A^m$ sur un facteur direct d'un foncteur induit depuis $A^{m+1}$.

On montre maintenant la deuxième assertion par récurrence sur $n$.
Le cas $n=0$ découle de ce qui précède et de la proposition~\ref{pr-psfn-ind}. Supposons maintenant que l'implication : {\em $\Delta(F)$ vérifie $\psf_{n-1}(r)$ $\Rightarrow$ $F$ vérifie $\psf_{n-1}(n+r)$} est satisfaite (pour tout $r$) et que $\Delta(F)$ est $\psf_n(m)$, avec $n>0$. Le cas $n=0$ montre que $F$ vérifie $\psf_0(m+1)$, donc il existe une suite exacte $0\to Y\to X\to F\to 0$ où $X$ est induit depuis $A^{m+1}$ (on utilise constamment, ici, la proposition~\ref{pr-psfn-ind}). Dans la suite exacte $0\to\Delta(Y)\to\Delta(X)\to\Delta(F)\to 0$, $\Delta(X)$ est induit depuis $A^{m+1}$, donc $\psf_\infty(m+1)$, et $\Delta(F)$ est $\psf_n(m+1)$, donc $\Delta(Y)$ est $\psf_{n-1}(m+1)$, par le troisième point de la proposition~\ref{pr-psfn-sec}. L'hypothèse de récurrence montre alors que $Y$ est $\psf_{n-1}(n+m+1)$. Appliquant maintenant le deuxième point de la proposition~\ref{pr-psfn-sec}, on tire de la suite exacte $0\to Y\to X\to F\to 0$ que $F$ est $\psf_n(n+m+1)$, ce qui achève la démonstration.
\end{proof}

Appliquant le deuxième point du lemme par récurrence sur le degré polynomial $d$, et notant qu'un foncteur constant vérifie $\psf_\infty(0)$, on obtient :

\begin{thm}\label{cor-psf} Soient $A$ un anneau et $F$ un foncteur de $\fct(\mathbf{P}(A);\E)$, polynomial de degré au plus $d$. Alors $F$ vérifie $\psf_i((i+1)d)$ pour tout $i\in\mathbb{N}$. En particulier, $F$ possède la propriété $psf_\infty$.
\end{thm}

On peut déduire de ce résultat, comme dans le lemme de Schwartz originel \cite[§\,10]{FLS}, la propriété $pf_\infty$ pour des foncteurs polynomiaux lorsque l'anneau $A$ est {\em fini} :

\begin{cor}\label{cor-anneaufini}
Soient $A$ un anneau {\em fini} et $F$ un foncteur de $\fct(\mathbf{P}(A);\E)$. Si $F$ est polynomial et à valeurs dans $\pf_\infty(\E)$, alors il vérifie la propriété $pf_\infty$. 
\end{cor}

\begin{proof}
Comme $A$ est fini, on peut appliquer le corollaire~\ref{cor-pfpsf} à la catégorie source $\mathbf{P}(A)$, ce qui permet de déduire le résultat du théorème~\ref{cor-psf}.
\end{proof}

\begin{rem} Supposons $A$ fini. Lorsque la catégorie $\E$ est localement noethérienne, $\fct(\mathbf{P}(A);\E)$ est également localement noethérienne \cite{PSam,SamSn}. Le corollaire~\ref{cor-anneaufini} en découle, par l'intermédiaire de la proposition~\ref{anc-rqc}. On notera toutefois que la noethérianité locale de $\fct(\mathbf{P}(A);\E)$ est un théorème beaucoup plus difficile à démontrer que les résultats de la présente section. 
\end{rem}

Si l'anneau $A$ n'est pas fini, il peut exister des foncteurs polynomiaux sur $\mathbf{P}(A)$ à valeurs $pf_\infty$ qui ne sont pas de présentation finie. Par exemple, dans $\F(A,K)$, où $K$ est un corps commutatif, \cite[théorème~13.8]{DTV} donne une condition nécessaire et suffisante pour que tous les foncteurs polynomiaux à valeurs de dimensions finies soient de présentation finie. Elle n'est pas vérifiée lorsque $A$ est un anneau de polynômes en une infinité d'indéterminées sur $\mathbb{Z}$ (en effet, il existe alors un $(K,A)$-bimodule simple, de dimension finie sur $K$, qui n'est pas de présentation finie). Nous donnerons à la section~\ref{schw2} des conditions suffisantes pour que tout foncteur polynomial à valeurs de dimensions finies de $\F(A,K)$ possède la propriété $pf_\infty$.

\medskip

Le théorème~\ref{cor-psf} combiné au corollaire~\ref{cor-ext-monoides} fournit le résultat suivant.

\begin{cor}\label{cor-bp1} Soient $F$ et $G$ des foncteurs de $\F(A,k)$. Supposons que $F$ est polynomial de degré au plus $d\in\mathbb{N}$. Alors le morphisme naturel
\[\mathrm{Ext}^j_{\F(A,k)}(F,G)\to\mathrm{Ext}^j_{k[\M_n(A)]}(F(A^n),G(A^n))\]
est injectif pour $n\ge (j+1)d$ et bijectif pour $n\ge (j+2)d$.
\end{cor}

Le corollaire~\ref{cor-bp1} améliore et généralise plusieurs résultats connus.
Lorsque $A=k$ est un corps fini, une forme faible du corollaire \ref{cor-bp1} (la bijectivité pour $n$ assez grand) est attestée par Kuhn \cite[{\em Theorem}~3.10]{Ku2} comme corollaire du lemme de Schwartz. Un autre cas particulier ($A=k$, pouvant être ici un anneau quelconque, et $F$ égal au foncteur d'inclusion de $\mathbf{P}(A)$ dans $A\Md$) avait été démontré antérieurement par Jibladze et Pirashvili \cite[proposition~2.13]{JP}, sous la condition plus restrictive $n\ge 2^j$.

\subsection{Application : stabilité homologique pour les monoïdes de matrices}\label{ssect-BP}

Les résultats du paragraphe précédent permettent d'améliorer les bornes de stabilité -- exponentielles -- obtenues par Betley et Pirashvili \cite{BP}, leur substituant des bornes linéaires, pour l'homologie des monoïdes de matrices. Pour $n\in\mathbb{N}$, et $B$ un bifoncteur de $\F(\mathbf{P}(A)^\op\times\mathbf{P}(A);\mathbb{Z})$, $B(A^n,A^n)$ possède une structure naturelle de bimodule sur l'anneau $\mathbb{Z}[\M_n(A)]$ du monoïde multiplicatif des matrices $n\times n$ à coefficients dans $A$, l'action à droite (resp. gauche) de ce monoïde provenant de la fonctorialité par rapport à la première (resp. deuxième) variable. L'homologie de Hochschild $HH_*(\mathbb{Z}[\M_n(A)];B(A^n,A^n))$ de ce bimodule s'identifie à celle de la restriction de $B$ à la sous-catégorie pleine de $\mathbf{P}(A)$ constituée du $A$-module libre $A^n$.

On dispose ainsi, par restriction, de morphismes naturels de groupes abéliens gradués
$$\xymatrix{\cdots\ar[r] & HH_*(\mathbb{Z}[\M_n(A)];B(A^n,A^n))\ar[r]\ar[rd] & HH_*(\mathbb{Z}[\M_{n+1}(A)];B(A^{n+1},A^{n+1}))\ar[r]\ar[d] & \cdots\\
& & HH_*(\mathbf{P}(A);B) & 
}$$
où la flèche horizontale est induite par la projection $A^{n+1}\twoheadrightarrow A^n$ sur les $n$ premiers facteurs (sur la première variable de $B$) et l'inclusion $A^n\hookrightarrow A^{n+1}$ des $n$ premiers facteurs (sur la seconde). Il s'agit d'un \emph{morphisme de stabilisation}, dont l'analogue pour les \emph{groupes} linéaires, et plus généralement pour de nombreuses familles infinies de groupes, a été très étudié, par des méthodes différentes, à partir de travaux inauguraux de Quillen (non publiés) --- cf. le survol systématique du sujet par Randal-Williams et Wahl \cite{RWW}. 

On dispose d'un diagramme similaire en cohomologie.

La proposition ci-dessous, qui est une variante du corollaire~\ref{cor-bp1}, constitue donc un résultat de stabilité (co)homologique pour les monoïdes $\M_n(A)$.

\begin{pr}\label{pr-BP1} Soient $A$ un anneau et $B$ un bifoncteur de $\F(\mathbf{P}(A)^\op\times\mathbf{P}(A);\mathbb{Z})$. Supposons qu'il existe un entier $d\in\mathbb{N}$ tel que $B$ est polynomial de degré au plus $d$ par rapport à la première variable (c'est-à-dire que, pour tout $V$, le foncteur $B(-,V)$ est polynomial de degré au plus $d$), ou par rapport à la seconde. Alors le morphisme canonique
$$HH_i(\mathbb{Z}[\M_n(A)];B(A^n,A^n))\to HH_i(\mathbf{P}(A);B)$$
est surjectif pour $n\ge d.(i+1)$ et bijectif pour $n\ge d.(i+2)$, et le morphisme canonique
$$HH^i(\mathbf{P}(A);B)\to HH^i(\mathbb{Z}[\M_n(A)];B(A^n,A^n))$$
est injectif pour $n\ge d.(i+1)$ et bijectif pour $n\ge d.(i+2)$.
\end{pr}

\begin{proof} Pour le morphisme de comparaison homologique, on combine le théorème~\ref{cor-psf} aux propositions~\ref{pr-nsuprod} et~\ref{pr-supp-HH}.

La propriété cohomologique s'en déduit comme suit. Le foncteur de dualité $D : \F(\C^\op\times\C;\mathbb{Z})^\op\to\F(\C^\op\times\C;\mathbb{Z})$ (pour une catégorie source $\C$ quelconque) donné par $(DF)(U,V)=\Hom_\mathbb{Z}(F(V,U),\mathbb{Q}/\mathbb{Z})$ est exact et fidèle, et donne lieu à des isomorphismes naturels $HH^i(\C;D(B))\simeq\Hom_\mathbb{Z}(HH_i(\C;B),\mathbb{Q}/\mathbb{Z})$. Ceci montre le résultat pour des bifoncteurs $B$ du type $D(B')$, où $B'$ vérifie la condition de polynomialité requise. Comme le morphisme canonique $B\to D(D(B))$ est injectif, on peut trouver une corésolution de tout bifoncteur $B$ sur $\C$ par des foncteurs de la forme $D(B')$, où les $B'$ vérifient la même condition de polynomialité que $B$, ce qui permet d'établir le cas général.
\end{proof}

\section{La propriété $pf_n$ pour les foncteurs additifs}\label{sec-add}

Rappelons que $\add(\A,\E)$ désigne la catégorie des foncteurs {\em additifs} de $\A$ vers $\E$. C'est, comme $\E$, une catégorie de Grothendieck. 
La catégorie $\add(\A;k)=\add(\A,k\Md)$ est engendrée par les objets projectifs de type fini $k\otimes_\mathbb{Z}\A(a,-)$ (représentant l'évaluation en $a$, par le lemme de Yoneda), où $a$ parcourt un squelette de $\A$.

La catégorie $\add(\A;k)$ est généralement beaucoup plus facile à étudier que $\F(\A;k)$ (par exemple, pour $\A=\mathbf{P}(A)$, la catégorie des foncteurs additifs est équivalente à celle des $(k,A)$-bimodules, par le théorème d'Eilenberg-Watts). Les propriétés de finitude d'un objet qui ne dépendent que de l'ensemble ordonné de ses sous-objets, telles la noethérianité, sont trivialement équivalentes, pour un foncteur additif, dans $\add(\A;k)$ et $\F(\A;k)$. Ce n'est toutefois pas le cas en général pour la propriété $pf_n$, comme nous le verrons dans l'exemple~\ref{ex-pasprespfn}.

Dans cette section, nous étudions comment l'inclusion $\add(\A;k)\hookrightarrow \F(\A;k)$ se comporte vis-à-vis de la propriété $pf_n$.

Précisément, nous montrons à la proposition \ref{pr-add-Q} que la propriété $pf_n$ est toujours préservée par l'inclusion si $k$ est une $\mathbb{Q}$-algèbre. Nous montrons aux théorèmes \ref{thm-add-Zp} et \ref{thm-add-Z} que si $k=\mathbb{Z}$ ou si $k$ contient un corps fini, la préservation de la propriété $pf_n$ est reliée à la propriété de transfert $(T_{n,p})$ étudiée dans la sous-section préliminaire qui suit.

\subsection{La propriété $T_{n,p}$}

Dans l'énoncé suivant, on note $_pV=\Hom_{\mathbb{Z}}(\mathbb{Z}/p,V)$ le sous-groupe de $p$-torsion d'un groupe abélien $V$.

\begin{defiprop}
Soient $n\in\mathbb{N}\cup \{\infty\}$ et $p$ un nombre premier. Les conditions suivantes sont équivalentes. 
\begin{enumerate}
\item[]
\begin{enumerate}
\item[$(T^1_{n,p})$] Pour tout $i\le n$, l'inclusion $\add(\A;\mathbb{Z}/p)\hookrightarrow\add(\A;\mathbb{Z})$ préserve les objets $pf_i$.
\item[$(T^2_{n,p})$] Pour tout $i\le n$, un foncteur de $\add(\A;\mathbb{Z}/p)$ appartient à la classe $\pf_i(\add(\A;\mathbb{Z}/p))$ si et seulement s'il appartient à $\pf_i(\add(\A;\mathbb{Z}))$.
\item[$(T^3_{n,p})$] Pour tout objet $a$ de $\A$, le foncteur $_p \A(a,-)$ appartient à la classe $\pf_{n-2}(\add(\A;\mathbb{Z}))$.
\item[$(T^4_{n,p})$] Pour tout objet $a$ de $\A$, le foncteur $_p \A(a,-)$ appartient à la classe $\pf_{n-2}(\add(\A;\mathbb{Z}/p))$.
\end{enumerate}
\end{enumerate}
Lorsqu'elles sont satisfaites, nous dirons que la petite catégorie additive $\A$ satisfait la propriété de transfert $(T_{n,p})$.
\end{defiprop}
\begin{proof}
L'équivalence entre $(T^1_{n,p})$ et $(T^3_{n,p})$ découle des suites exactes $0\to\negmedspace_p \A(a,-)\to\A(a,-)\xrightarrow{.p}\A(a,-)\to\A(a,-)\otimes_\mathbb{Z}\mathbb{Z}/p\to 0$, de la proposition~\ref{pr-pfn-sec} et de l'exemple~\ref{expfinf}(b). 

On montre ensuite par récurrence sur $n$ que les quatre conditions  sont équivalentes. Plus précisément, pour $n=0$ les quatre conditions sont vraies quelle que soit la catégorie $\A$, et sont donc trivialement équivalentes. Supposons démontré que les quatre conditions $(T^m_{i,p})$ sont équivalentes pour tout $i<n$. Alors pour tout $m\in\{1,2,3,4\}$ et tout $i<n$ on a l'implication:
$$(T^m_{n,p})\Rightarrow (T^m_{i,p})\Rightarrow \big((T^1_{i,p}) +(T^2_{i,p})+(T^3_{i,p}) +(T^4_{i,p})\big).$$
On obtient alors l'équivalence entre les quatre conditions $(T^m_{n,p})$ à l'aide des implications suivantes. 
\begin{align*}
&(T^2_{n,p})\Rightarrow (T^1_{n,p})\;; && \big((T^1_{n,p})+(T^2_{n-1,p})\big)\Rightarrow (T^2_{n,p})\;;\\
&\big((T^4_{n,p})+(T^1_{n-2,p})\big)\Rightarrow (T^3_{n,p})\;;
&&\big((T^3_{n,p})+(T^2_{n-2,p})\big)\Rightarrow (T^4_{n,p})\;.
\end{align*}
Parmi ces implications, $\big((T^1_{n,p})+(T^2_{n-1,p})\big)\Rightarrow (T^2_{n,p})$ provient de la proposition~\ref{pr-pfn-efonc}, qui s'applique à l'inclusion $\add(\A;\mathbb{Z}/p)\hookrightarrow\add(\A;\mathbb{Z})$ grâce à la proposition~\ref{rq-prestf}. Les autres implications sont immédiates.
\end{proof}

\begin{pr}\label{pr-ptt} La propriété $(T_{n,p})$ est vérifiée pour tout $n$ dans chacun des cas suivants:
\begin{enumerate}
\item\label{pt1} lorsque la catégorie $\add(\A;\mathbb{Z}/p)$ est localement noethérienne ;
\item\label{pt2} lorsque la catégorie $\add(\A,\mathbf{Ab})$ est localement noethérienne ;
\item\label{pt3} lorsque la catégorie $\A$ est $\mathbb{Z}/p$-linéaire ;
\item\label{pt4} lorsque les groupes abéliens de morphismes dans la catégorie $\A$ sont sans $p$-torsion ;
\item\label{pt5} lorsque la catégorie $\A$ possède des conoyaux. 
\end{enumerate}
\end{pr}

\begin{proof}
\begin{enumerate}
\item[(\ref{pt1})] La proposition~\ref{anc-rqc} montre que $(T^4_{n,p})$ est satisfaite.
\item[(\ref{pt2})] se ramène à (\ref{pt1}) car $\add(\A;\mathbb{Z}/p)$ est une sous-catégorie abélienne de $\add(\A,\mathbf{Ab})$ stable par sous-objet.
\item[(\ref{pt3})] : on a alors $_p \A(a,-)=\A(a,-)$ qui est projectif de type fini dans $\add(\A;\mathbb{Z})$, donc $pf_\infty$ (exemple~\ref{expfinf}(b)), d'où $(T^3_{n,p})$.
\item[(\ref{pt4})] : $(T^3_{n,p})$ est évidemment vérifié.
\item[(\ref{pt5})] : on a alors $_p \A(a,-)\simeq\A(a/p,-)$, où $a/p$ désigne le conoyau de $p.\mathrm{Id}_a$, d'où $(T^3_{n,p})$.
\end{enumerate}
\end{proof}

\begin{ex}\label{ex-pasT}
Si $\A=\mathbf{P}(A)$, l'inclusion $\add(\A;\mathbb{Z}/p)\hookrightarrow\add(\A;\mathbb{Z})$ s'identifie à la restriction $\Mdd (A\otimes_\mathbb{Z}\mathbb{Z}/p)\to\Mdd A$ le long de l'épimorphisme canonique $A\to A\otimes_\mathbb{Z}\mathbb{Z}/p$. Ainsi, $\mathbf{P}(A)$ possède la propriété $(T_{n,p})$ si et seulement si l'idéal $_pA$ des éléments de $A$ annulés par multiplication par $p$ est dans $\pf_{n-2}(\Mdd A)$.

Cet idéal peut ne pas être de type fini, comme l'illustre l'exemple de l'anneau $A=\mathbb{Z}\oplus V$, où $V$ est un $\mathbb{Z}/p$-espace vectoriel infini, avec la multiplication $(a,v).(b,w)=(ab,aw+bv)$. Dans ce cas $\mathbf{P}(A)$ ne possède pas la propriété $(T_{n,p})$ si $n>1$.
\end{ex}

\begin{rem}
L'inclusion $\add(\A;\mathbb{Z}/p)\hookrightarrow\add(\A;\mathbb{Z})$ ne préserve pas toujours la propriété $pf_n$ (sauf si $n\le 1$), cf. l'exemple \ref{ex-pasT}. La propriété $(T_{n,p})$ constitue donc une caractéristique intéressante d'une catégorie additive $\A$.
\`A l'inverse, l'inclusion $\F(\A;\mathbb{Z}/p)\hookrightarrow\F(\A;\mathbb{Z})$ préserve toujours les objets $pf_n$, et ne dit rien de la catégorie $\A$. Cette préservation résulte des suites exactes courtes 
$0\to \mathbb{Z}[\A(a,-)]\xrightarrow[]{p}\mathbb{Z}[\A(a,-)]\to \mathbb{Z}/p[\A(a,-)]\to 0$,
de l'exemple~\ref{expfinf}(b) et de la proposition~\ref{pr-pfn-efonc}.
\end{rem}

\begin{rem}
Les travaux d'Auslander \cite{Aus82} donnent de nombreux exemples de petites catégories \emph{abéliennes} $\A$ telles que $\add(\A;\mathbb{Z}/p)$ n'est pas localement noethé\-rien\-ne (par exemple, la catégorie des modules à gauche de type fini sur une $\mathbb{Z}/p$-algèbre finie de type de représentation infini), qui entrent donc dans le dernier point de la proposition~\ref{pr-ptt} mais pas dans les deux premiers.
\end{rem}

\subsection{Résultats principaux}

Nos résultats de transfert de la propriété $pf_n$ de $\add(\A;k)$ à $\F(\A;k)$ (tirés essentiellement des premières versions prépubliées de \cite{DTV}) reposent sur la \emph{construction cubique de MacLane} \cite[§\,12]{EMLq} (on pourra aussi consulter la présentation plus récente de Johnson-McCarthy \cite[§\,6]{JMq}). Elle consiste en un complexe de chaînes $Q_*$ d'endofoncteurs des groupes abéliens dont l'homologie, évaluée sur $V$, calcule (fonctoriellement en $V$) l'homologie stable (à coefficients entiers) des espaces d'Eilenberg-MacLane associés à $V$.
La proposition suivante en résume les propriétés utiles.

\begin{pr}[Eilenberg-MacLane, Cartan, Dold-Puppe...]\label{cubique}
\begin{enumerate}
\item[(1)] Pour tout $n\in\mathbb{N}$, $Q_n$ est un facteur direct du foncteur $V\mapsto\mathbb{Z}[V^{2^n}]$.
\item[(2)] L'homologie de $Q_*$ en degré nul est isomorphe au foncteur identité.
\item[(3)] L'homologie de $Q_*$ est nulle en degré $1$ et isomorphe en chaque degré supérieur ou égal à $2$ à une somme directe finie de foncteurs du type  $-{\otimes}_{\mathbb{Z}}\mathbb{Z}/p$ ou $\Hom_{\mathbb{Z}}(\mathbb{Z}/p,-)$ pour des nombres premiers $p$.
Le plus bas degré où le foncteur $-{\otimes}_{\mathbb{Z}}\mathbb{Z}/p$, resp. $\Hom_{\mathbb{Z}}(\mathbb{Z}/p,-)$, apparaît comme facteur direct de l'homologie de $Q_*$ est $2p-2$, resp. $2p-1$.
\end{enumerate}
\end{pr}

\begin{proof} L'assertion (1) découle de la définition explicite de la construction cubique. Les deux autres résultent du calcul par Cartan \cite[exposé~11, §\,6, Théorème~2]{Cartan} de l'homologie stable des espaces d'Eilenberg-MacLane.
\end{proof}

Par la formule des coefficients universels, on en déduit le résultat suivant. 

\begin{cor}\label{cubique-corps} Si l'anneau $k$ contient un corps premier, il existe un complexe de chaînes $Q_*^k$ de foncteurs $\mathbf{Ab}\to k\Md$ possédant les propriétés suivantes.
\begin{enumerate}
\item[(1)] Pour tout $n\in\mathbb{N}$, $Q_n^k$ est un facteur direct du foncteur $V\mapsto k[V^{2^n}]$.
\item[(2)] L'homologie de $Q_*^k$ en degré nul est isomorphe à $V\mapsto V\otimes_\mathbb{Z} k$.
\item[(3)] Si $k$ est de caractéristique nulle, l'homologie de $Q_*^k$ est nulle en degrés strictement positifs.
\item[(4)] Si $k$ est de caractéristique $p>0$, l'homologie de $Q_*^k$ est isomorphe en chaque degré non nul à une somme directe finie de foncteurs du type  $-{\otimes}_{\mathbb{Z}}\mathbb{Z}/p$ ou $\Hom_{\mathbb{Z}}(\mathbb{Z}/p,-)$.
Le plus bas degré strictement positif où le foncteur $-{\otimes}_{\mathbb{Z}}\mathbb{Z}/p$, resp. $\Hom_{\mathbb{Z}}(\mathbb{Z}/p,-)$, apparaît comme facteur direct de l'homologie de $Q_*$ est $2p-3$, resp. $2p-2$.
\end{enumerate}
\end{cor}

\begin{pr}\label{pr-add-Q} Supposons que $k$ est une $\mathbb{Q}$-algèbre. Soit $n\in\mathbb{N}\cup\{\infty\}$. Alors l'inclusion $\add(\A;k)\hookrightarrow\F(\A;k)$ préserve les objets $pf_n$.
\end{pr}

\begin{proof} D'après la proposition~\ref{pr-pfn-efonc}, il suffit de montrer que les foncteurs additifs $\A(a,-)\otimes_\mathbb{Z}k$ appartiennent à $\pf_\infty(\F(\A;k))$. Cela résulte du corollaire~\ref{cubique-corps}, qui montre que le complexe $Q^k_*\circ\A(a,-)$ est une résolution projective dans $\F(\A;k)$ de $\A(a,-)\otimes_\mathbb{Z} k$, dont tous les termes sont de type fini.
\end{proof}

Le théorème suivant est l'analogue de la proposition~\ref{pr-add-Q} en caractéristique positive. Il fournit également une nouvelle caractérisation de la propriété de transfert $(T_{n,p})$.

\begin{thm}\label{thm-add-Zp} Soit $n\in\mathbb{N}\cup\{\infty\}$. Supposons que $k$ est un anneau non nul de caractéristique $p$ première. Alors $\A$ vérifie la propriété $(T_{n,p})$ si et seulement si l'inclusion $\add(\A;k)\hookrightarrow\F(\A;k)$ préserve la propriété $pf_i$ pour tout $i\le n+2p-3$.
\end{thm}
\begin{proof} Supposons que $\A$ vérifie la propriété $(T_{n,p})$, et montrons que l'inclusion $\add(\A;k)\hookrightarrow\F(\A;k)$ préserve la propriété $pf_i$ pour $i\le n+2p-3$. D'après les propositions~\ref{pr-pfn-efonc} et \ref{pr-chbpfn}, il suffit pour cela de montrer que la propriété 
\begin{description}
\item[\quad$(HR_i)$\quad]
$\A(a,-)\otimes_\mathbb{Z}\mathbb{Z}/p$ est dans $\pf_i(\F(\A;\mathbb{Z}/p))$ pour tout $a\in {\rm Ob}\,\A$
\end{description}
est vraie pour tout $i\le n+2p-3$. 
Nous la démontrons par récurrence sur $i$. La propriété $(HR_0)$ est vérifiée car $\A(a,-)\otimes_\mathbb{Z}\mathbb{Z}/p$ est un quotient de $\mathbb{Z}/p[\A(a,-)]$. Supposons $(HR_i)$ vérifiée, avec $i<n+2p-3$. La proposition~\ref{pr-pfn-efonc} montre alors que l'inclusion $\add(\A;\mathbb{Z}/p)\hookrightarrow\F(\A;\mathbb{Z}/p)$ préserve la propriété $pf_j$ pour $j\le i$. Comme $\A$ vérifie de plus $(T_{n,p})$, les foncteurs $_p\A(a,-)$ sont dans $\pf_{m}(\F(\A;\mathbb{Z}/p))$ pour $m=\min\{i,n-2\}$. 
Considérons maintenant le complexe $Q_*^{\mathbb{Z}/p}\circ \A(a,-)$ de la proposition~\ref{cubique-corps}. Il est constitué de foncteurs projectifs de type fini de $\F(\A;\mathbb{Z}/p)$. Son homologie en degré nul est isomorphe à $\A(a,-)\otimes_{\mathbb{Z}}\mathbb{Z}/p$, son homologie est nulle en degrés $j$ pour $0<j<2p-3$, son homologie de degré $2p-3$ est dans $\pf_{i-(2p-3)}(\F(\A;\mathbb{Z}/p))$, et en degrés $j\ge 2p-4$, son homologie est constituée de foncteurs de $\pf_m(\F(\A;\mathbb{Z}/p))$ avec $m\ge i-j$. On en déduit alors $(HR_{i+1})$ par le corollaire~\ref{cor-pfn-compl}.

Réciproquement, supposons que l'inclusion $\add(\A;k)\hookrightarrow\F(\A;k)$ préserve la propriété $pf_j$ pour tout $j\le n+2p-3$. D'après la proposition \ref{pr-chbpfn}, cela implique que les foncteurs $\A(a,-)\otimes_{\mathbb{Z}}\mathbb{Z}/p$ sont $pf_{n+2p-3}$, donc que l'inclusion $\add(\A;\mathbb{Z}/p)\hookrightarrow\F(\A;\mathbb{Z}/p)$ préserve la propriété $pf_j$ pour tout $j\le n+2p-3$. Nous allons montrer que $\A$ vérifie $(T_{n,p})$ par récurrence. Comme toute catégorie additive, $\A$ vérifie $(T_{0,p})$. Supposons que  $\A$ vérifie $(T_{i,p})$, avec $i<n$. Alors les foncteurs $_p\A(a,-)$ sont dans $\pf_{i-2}(\F(\A;\mathbb{Z}/p))$.
Soit $C_*$ le complexe $Q_*^{\mathbb{Z}/p}\circ \A(a,-)$ ; notons $Z_j$ le foncteur des cycles de degré $j$ de ce complexe, et $B_j$ le foncteur des bords de degré $j$. En utilisant les suites exactes 
\begin{align*}
&0\to Z_{2p-3}\to C_{2p-3}\to \cdots \to C_0\to H_0(C)\to 0\;,\\
&0\to B_{2p-3}\to Z_{2p-3}\to H_{2p-3}(C)\to 0\;,\\
&0\to Z_{2p-2}\to C_{2p-2}\to B_{2p-3}\to 0\;,
\end{align*}
le fait que les foncteurs $C_j$ sont projectifs de type fini (donc $pf_\infty$),
le fait que $H_0(C)$ et $H_{2p-3}(C)$ sont des sommes directes finies de copies de $\A(a,-)\otimes_\mathbb{Z}\mathbb{Z}/p$ (donc $pf_{n+2p-3}$) et le troisième point de la proposition \ref{pr-pfn-sec}, on obtient que le foncteur $Z_{2p-2}$ appartient à $\pf_{n-2}(\F(\A;k))$.
Par ailleurs, le complexe tronqué décalé
$T_*=C_{*+2p-1}$ pour $*\ge 0$ est un complexe de projectifs de type fini, dont les groupes d'homologie en degré non nul sont $pf_{i-2}$. Le corollaire~\ref{cor-pfn-compl} montre donc que son homologie de degré nul, le foncteur $B_{2p-2}$, appartient à $\pf_{i-1}(\F(\A;k))$.
On déduit alors du deuxième point de la proposition \ref{pr-pfn-sec} que le foncteur $H_{2p-2}(C)$, donc son facteur direct $_p\A(a,-)$, appartient à $\pf_{m}(\F(\A;k))$ pour $m=\min\{i,n-2\}$, donc pour $m=i-1$. Ainsi $\A$ vérifie $(T_{i+1,p})$.
\end{proof}

\begin{ex}\label{ex-pasprespfn} Comme la propriété $(T_{1,p})$ est toujours vérifiée,  $\add(\A;k)\hookrightarrow\F(\A;k)$ préserve la propriété $pf_i$ pour tout $i\le 2p-2$, lorsque $k$ est de caractéristique $p$ première. En revanche, pour tout $p$ premier, il existe un anneau $A$ tel que $\mathbf{P}(A)$ ne satisfait pas $(T_{2,p})$ (exemple~\ref{ex-pasT}), de sorte que $\add(\mathbf{P}(A);\mathbb{F}_p)\hookrightarrow\F(\mathbf{P}(A);\mathbb{F}_p)$ ne préserve pas $pf_{2p-1}$.
\end{ex}

En se basant sur la proposition \ref{cubique} et raisonnant comme pour le sens direct du théorème \ref{thm-add-Zp}, on obtient le résultat suivant.

\begin{thm}\label{thm-add-Z} Soit $n\in\mathbb{N}\cup\{\infty\}$. Si $\A$ vérifie la propriété $(T_{n-2p+2,p})$ pour tout $p$ premier, alors  l'inclusion $\add(\A;\mathbb{Z})\hookrightarrow\F(\A;\mathbb{Z})$ préserve la propriété $pf_n$.
\end{thm}

\section{La propriété $pf_n$ pour les foncteurs polynomiaux}\label{schw2}

Dans la section~\ref{s-psf}, nous avons montré que les foncteurs polynomiaux sur $\mathbf{P}(A)$ (où $A$ est un anneau) vérifient la propriété $psf_\infty$, généralisant le lemme de Schwartz \cite[Proposition~10.4]{FLS}. L'objectif de la présente section est d'établir une autre généralisation du lemme de Schwartz, donnant cette fois la propriété $pf_\infty$ pour des foncteurs polynomiaux. 

Nous obtenons en fait deux résultats en ce sens. Le premier, valable lorsque l'anneau de coefficients au but $k$ est un corps, est le théorème~\ref{th-pol-pfn}. Il repose sur les résultats de \cite{DTV} qui indiquent que, sous des hypothèses raisonnables, un foncteur polynomial simple de $\F(\A;k)$ s'écrit comme un produit tensoriel de foncteurs du type $E\circ A$ où $E$ est un foncteur polynomial de $\F(k,k)$ et $A$ un foncteur additif dans $\add(\A;k)$. Pour établir le théorème \ref{th-pol-pfn}, nous étudions tout d'abord dans les sections \ref{ss-pol} et \ref{ss-polstrict} comment la propriété $pf_n$ est préservée par postcomposition par un foncteur polynomial $E$.

Le deuxième résultat, valable pour $k=\mathbb{Z}$, est énoncé au théorème \ref{th-polnlpf} et au corollaire \ref{cor-polpfi}. On l'obtient en analysant la filtration polynomiale des foncteurs.  

Nous appliquons ensuite nos résultats pour obtenir de nouveaux résultats de finitude pour l'homologie de Hochschild des catégories et les groupes $\GL(A)$.

\begin{nota} On désigne par $\pol_d(\A,\E)$ (resp. $\pol_d(\A;k)$, et $\pol_d(A,k)$, si $A$ est un anneau) la sous-catégorie pleine des foncteurs polynomiaux de degré au plus $d$ de $\fct(\A,\E)$ (resp. $\F(\A;k)$, $\F(A,k)$).
\end{nota}
Les rappels utiles sur la notion de foncteur polynomial, qui remonte à Eilenberg-MacLane \cite[chapitre~II]{EML}, peuvent se trouver par exemple dans \cite[§\,2]{DTV}.

Si $F$ est un foncteur de $\F(k,k)$, nous noterons $\tilde{F}$ son extension de Kan à gauche le long de l'inclusion pleinement fidèle de $\mathbf{P}(k)$ dans la catégorie des $k$-modules {\em plats}. Le foncteur $F\mapsto\tilde{F}$ est \emph{exact} car $\tilde{F}(M)$ s'obtient par colimite des $F(V)$ pour $V$ dans $\mathbf{P}(k)$ muni d'un morphisme $k$-linéaire vers $V$, et cette colimite est \emph{filtrante} pour $V$ plat.

\subsection{Postcomposition par un foncteur polynomial}\label{ss-pol}

On rappelle que l'anneau $k$ est dit {\em solide} si la multiplication $k\otimes_\mathbb{Z}k\to k$ est un isomorphisme.

\begin{lm}\label{lm-res-fpoldf} Supposons que l'anneau $k$ est solide et noethérien. Soient $d\in\mathbb{N}$ et $F$ un foncteur de $\F(k,k)$ polynomial de degré au plus $d$, et dont les valeurs sont des $k$-modules de type fini. Alors il existe un complexe de chaînes $C_\bullet$ de $\F(k,k)$, d'homologie notée $H_\bullet$, possédant les propriétés suivantes :
\begin{enumerate}
\item pour tout $i\in\mathbb{N}$, le foncteur $C_i$ est une somme directe finie de copies de la $d$-ième puissance tensorielle $T^d$ sur $k$ ;
\item pour tout $i>0$, $H_i$ est un foncteur polynomial de degré \emph{strictement} inférieur à $d$, à valeurs dans les $k$-modules de type fini ;
\item il existe un morphisme $H_0\to F$ dont le noyau et le conyau sont des foncteurs polynomiaux de degré \emph{strictement} inférieur à $d$ à valeurs dans les $k$-modules de type fini.
\end{enumerate}
\end{lm}

\begin{proof} Comme l'anneau $k$ est solide, le foncteur
$$\F(k,k)\to k[\Si_d]\Md\qquad X\mapsto cr_d(X)(k,\dots,k)$$
(où $\Si_d$ désigne le groupe symétrique sur $d$ lettres et $cr_d$ le $d$-ième effet croisé, comme dans \cite{DTV}) induit une équivalence $\pol_d(k,k)/\pol_{d-1}(k,k)\xrightarrow{\simeq} k[\Si_d]\Md$, par un résultat classique de Pirashvili \cite{Pira88}. On en déduit aisément que le complexe $C_\bullet:=T^d\underset{k[\Si_d]}{\otimes}R_\bullet$, où $R_\bullet$ est une résolution du $k[\Si_d]$-module $cr_d(F)(k,\dots,k)$ par des modules libres de rang fini (une telle résolution existe car $F$ est à valeurs de type fini sur $k$ et que cet anneau est noethérien), convient.
\end{proof}

\begin{pr}\label{pr-postcompo-solide} Supposons que $\C$ possède des coproduits finis et que l'anneau $k$ est solide et noethérien. Soient $n\in\mathbb{N}\cup\{\infty\}$, $F$ un foncteur polynomial de $\F(k,k)$ à valeurs dans les $k$-modules de type fini, et $X$ un foncteur de $\F(\C;k)$ à valeurs plates. Si $X$ appartient à $\pf_n(\F(\C;k))$, alors il en est de même pour $\tilde{F}\circ X$.
\end{pr}

\begin{proof} Soit $d$ le degré polynomial de $F$ ; par récurrence, on peut supposer que $\tilde{G}\circ X$ vérifie $pf_n$ si $G$ est un foncteur polynomial de degré $<d$, à valeurs de type fini, de $\F(k,k)$. Considérons un complexe $C_\bullet$ donné par le lemme~\ref{lm-res-fpoldf}. Pour tout $i\in\mathbb{N}$, le foncteur $\tilde{C}_i\circ X$ possède la propriété $pf_n$, grâce au premier point de ce lemme et à la proposition~\ref{pr-ptens-int}. Pour $i>0$, l'hypothèse de récurrence et le deuxième point du lemme ~\ref{lm-res-fpoldf} montrent que $\tilde{H}_i\circ X$ vérifie $pf_n$. L'exactitude de $F\mapsto\tilde{F}$ et de la précomposition par $X$, combinée au corollaire~\ref{cor-pfn-compl}, montre que $\tilde{H}_0\circ X$ est $pf_n$. Utilisant la dernière assertion du lemme ~\ref{lm-res-fpoldf} et l'hypothèse de récurrence, on en tire un morphisme  $\tilde{H}_0\circ X\to\tilde{F}\circ X$ dont source, noyau et conoyau vérifient  la propriété $pf_n$. La proposition~\ref{pr-pfn-sec} implique que son but $\tilde{F}\circ X$ vérifie également $pf_n$, d'où la conclusion.
\end{proof}

\begin{rem}\label{rq-vpplates} On ne peut pas s'affranchir en général de l'hypothèse de platitude des valeurs de $X$. L'exemple~\ref{ex-prodtens-pasplat} donne un cas particulier de ce phénomène lorsque $k=\mathbb{Z}$ et que $F$ est la deuxième puissance tensorielle.
\end{rem}

\subsection{Postcomposition par un foncteur polynomial strict}\label{ss-polstrict}

Afin de lever l'hypothèse, restrictive (cf. la description de tous les anneaux solides par Bousfield-Kan \cite[proposition~3.5]{BK-solide}), de solidité de $k$ dans la proposition~\ref{pr-postcompo-solide}, nous sommes conduits à considérer des foncteurs polynomiaux \emph{stricts}, qui permettent de se ramener par changement de base au but au cas de l'anneau (solide) des entiers.

Soit $d\in\mathbb{N}$. On rappelle que la catégorie, qu'on désigne par $\PP_d(k)$ (qui est notée $\mathbf{Rep}\;\Gamma^d_k$ dans le survol \cite{Kr-P} auquel on renvoie pour plus de détails sur cette notion qui remonte à Friedlander-Suslin \cite{FS}) des foncteurs polynomiaux stricts\,\footnote{Contrairement à certains auteurs, nous autorisons les foncteurs polynomiaux stricts à prendre leurs valeurs dans les $k$-modules arbitraires (pas seulement projectifs de type fini).} homogènes de poids $d$ sur $k$ est la catégorie des foncteurs $k$-linéaires $\Gamma^d_k(\mathbf{P}(k))\to k\Md$, où $\Gamma^d_k(\mathbf{P}(k))$ désigne la catégorie avec les mêmes objets que $\mathbf{P}(k)$ et pour morphismes $\big(\Gamma^d_k(\mathbf{P}(k))\big)(U,V):=\Gamma^d_k\big(\mathbf{P}(k)(U,V)\big)$, où $\Gamma^d_k$ (notée simplement $\Gamma^d$ s'il n'y a pas d'ambiguïté) désigne la $d$-ième puissance divisée (sur $k$), c'est-à-dire les invariants sous l'action du groupe symétrique $\Si_d$ de la $d$-ième puissance tensorielle sur $k$. La catégorie $\PP_d(k)$ est une catégorie de Grothendieck $k$-linéaire.

\begin{pr}\label{pr-Pln} Si l'anneau $k$ est noethérien, alors la catégorie $\PP_d(k)$ est localement noethérienne pour tout $d\in\mathbb{N}$.
\end{pr}

\begin{proof}
La catégorie $\PP_d(k)$ est équivalente à la catégorie des modules à gauche sur l'algèbre de Schur $S_{d,n}(k):=\mathrm{End}_{k[\Si_d]}((k^n)^{\otimes d})$ pour $n\ge d$ \cite[{\em Theorem}~2.10]{Kr-P}. Comme $S_{d,n}(k)$ est une $k$-algèbre finie, elle est noethérienne comme $k$, d'où le résultat.
\end{proof}

 On dispose d'un foncteur d'oubli exact et fidèle $\PP_d(k)\to\F(k,k)$ dont l'image est incluse dans les foncteurs polynomiaux de degré au plus $d$ ; par abus, on note de la même façon un foncteur de $\PP_d(k)$ et son image dans $\F(k,k)$.

\begin{pr}\label{pr-postcompo-P} Supposons que la catégorie $\C$ possède des coproduits finis et que l'anneau $k$ est noethérien. Soient $d\in\mathbb{N}$, $n\in\mathbb{N}\cup\{\infty\}$, $E$ un foncteur de type fini de $\PP_d(k)$ et $X$ un foncteur de $\pf_n(\F(\C;k))$. On suppose que $X$ prend ses valeurs dans les $k$-modules plats. Alors $\tilde{E}\circ X$ appartient à $\pf_n(\F(\C;k))$.
\end{pr}

\begin{proof} Les propositions~\ref{pr-Pln} et~\ref{anc-rqc} montrent que $E$ appartient à $\pf_\infty(\PP_d(k))$. Comme $\PP_d(k)$ est engendrée \cite[page.~1000]{Kr-P} par les foncteurs projectifs de type fini $\Gamma^{d,V}:=\Gamma^d(\Hom_k(V,-))$, où $V$ est un objet de $\mathbf{P}(k)$, l'exactitude de $E\mapsto\tilde{E}$ et les propositions~\ref{pr-pres-concr} et~\ref{anc-rqc} montrent qu'il suffit de voir que $\Gamma^d\circ X$ vérifie la propriété $pf_n$, quitte à remplacer $X$ par $\Hom_k(V,-)\circ X$ (qui est $pf_n$ comme $X$).

Le caractère $pf_n$ de $X$, combiné à la proposition~\ref{pr-pres-concr} et à la correspondance de Dold-Kan montrent qu'il existe un objet simplicial $R_\bullet$ de la catégorie $\F(\C;k)$ tel que $R_m$ est, pour tout $m\le n$, une somme directe finie de foncteurs projectifs $P^t_\C$ avec $t\in\mathrm{Ob}\,\C$, dont l'homotopie est isomorphe à $X$ en degré nul et triviale en degré supérieur. Comme $R_\bullet$ et son homotopie sont $k$-plats, on en déduit que $\Gamma^d\circ R_\bullet$ a une homotopie isomorphe à $\Gamma^d\circ X$ en degré nul, et triviale en degré supérieur. Or, pour $m\le n$, $\Gamma^d\circ R_m$ est isomorphe à $\Big(\Gamma^d_\mathbb{Z}\circ \big(\bigoplus_{j=1}^r \mathbb{Z}[\C(t_j,-)]\big) \Big)\otimes_\mathbb{Z} k$ pour des objets $t_1,\dots,t_r$ appropriés de $\C$. Comme $\mathbb{Z}$ est un anneau solide noethérien, la proposition~\ref{pr-postcompo-solide} montre que $\Gamma^d_\mathbb{Z}\circ \big(\bigoplus_{j=1}^r \mathbb{Z}[\C(t_j,-)]\big)$ appartient à $\pf_\infty(\F(\C;\mathbb{Z}))$. Comme il est à valeurs $\mathbb{Z}$-plates, cela implique par la proposition~\ref{pr-chbpfn} que $\Gamma^d\circ R_m$ appartient à $\pf_\infty(\F(\C;k))$ pour $m\le n$, de sorte que la correspondance de Dold-Kan et le corollaire~\ref{cor-pfn-compl} permettent de conclure.
\end{proof}

\begin{rem}
On peut s'affranchir de l'hypothèse de noethérianité de $k$, si l'on suppose que $E$ appartient à $\pf_n(\Pp_d(k))$.
\end{rem}

La proposition~\ref{pr-postcompo-P} admet la réciproque partielle suivante.

\begin{pr}\label{pr-recip-postcompo} On suppose que $k$ est un corps. Soient $E$ un foncteur polynomial strict non constant sur $k$, $X$ un foncteur de $\F(\A;k)$ et $n\in\mathbb{N}\cup\{\infty\}$. Si le foncteur $\tilde{E}\circ X$ appartient à $\pf_n(\F(\A;k))$, alors il en est de même de $X$.
\end{pr}

\begin{proof} Le terme constant $X(0)$ de $X$ est de dimension finie sur $k$, car le foncteur constant en $\tilde{E}(X(0))$ est facteur direct de $\tilde{E}\circ X$, et donc dans $\pf_n(\F(\A;k))$, autrement dit $\tilde{E}(X(0))$ est de dimension finie, ce qui implique la même propriété pour $X(0)$ puisque $E$ est non constant. Comme $\tilde{E}\circ\bar{X}$ (où l'on a décomposé $X\simeq\bar{X}\oplus X(0)$) est facteur direct de $\tilde{E}\circ X$, et donc dans $\pf_n(\F(\A;k))$, on peut supposer que $X$ est réduit (i.e. $X(0)=0$).

Soit $d\in\mathbb{N}^*$ le degré au sens d'Eilenberg-MacLane de $E$ (qui peut être strictement inférieur à son poids, ou degré polynomial strict). Alors $cr_d(E)$ est un multifoncteur polynomial strict en $d$ variables sur les $k$-espaces vectoriels, qui est non nul et additif par rapport à chaque variable. C'est donc \cite[proposition~B.3]{Touze} une somme directe non vide de foncteurs du type $(V_1,\dots,V_d)\mapsto V_1^{(i_1)}\otimes\dots\otimes V_d^{(i_d)}$, où les exposants entre parenthèses désignent des itérations convenables de la torsion de Frobenius (ou l'identité si $k$ est de caractéristique nulle). Le lemme~\ref{lm-crfd} ci-dessous et le fait que le foncteur $cr_d : \F(\A;k)\to\F(\A^d;k)$ préserve la propriété $pf_n$ (car c'est un facteur direct de la précomposition par la somme directe itérée $\A^d\to\A$, adjointe de chaque côté à la diagonale itérée, on peut donc appliquer la proposition~\ref{pr-precompf}), on en déduit qu'il existe $(i_1,\dots,i_d)$ tel que le foncteur $X^{(i_1)}\otimes\dots\otimes X^{(i_d)}$ appartienne à $\pf_n(\F(\A;k))$. Par la proposition~\ref{prptrec2}, on en déduit que $X^{(i)}$ vérifie $pf_n$ pour un $i\in\mathbb{N}$. Il en est de même pour $X$ par la proposition~\ref{pr-chbpfn}.
\end{proof}

Le résultat élémentaire suivant est classique, il suit par exemple de la démon\-stra\-tion de la proposition~1.6 de \cite{JM-cotriple}.

\begin{lm}\label{lm-crfd} Si $F$ est un foncteur réduit, alors $cr_d(E)\circ F^{\times d}$ est facteur direct de $cr_d(E\circ F)$.
\end{lm}

\begin{rem} La proposition~\ref{pr-recip-postcompo} demeure valide (et se montre de façon analogue) lorsque $E$ est un foncteur polynomial (non nécessairement strict) non constant de $\F(k,k)$, si $k$ est une extension finie de son sous-corps premier.
\end{rem}

\begin{rem} L'énoncé et la démonstration de la proposition~\ref{pr-recip-postcompo} procèdent d'un esprit très analogue aux résultats de Kuhn \cite[{\em Theorem}~4.8 et {\em Lemma}~4.12]{Ku3}. Kuhn utilise le foncteur décalage plutôt que les effets croisés, méthode qui pourrait s'appliquer dans notre situation seulement avec des hypothèses supplémentaires sur $\A$ : les foncteurs de décalage de $\F(\A;k)$ préservent la propriété $pf_n$ si et seulement si tous les groupes abéliens de morphismes $\A(x,y)$ sont {\em finis}.
\end{rem}

\subsection{Foncteurs polynomiaux à valeurs dans des espaces vectoriels de dimensions finies}\label{ssect-fpcor}

\begin{pr}\label{pr-pfnStein} Supposons que $k$ est un corps, $n$ un élément de $\mathbb{N}\cup\{\infty\}$ et $F$ un foncteur de $\F(\A;k)$ possédant une décomposition
\begin{equation}\label{eq-decSt}
F\simeq\bigotimes_{i=1}^r\tilde{E}_i\circ\rho_i
\end{equation}
où les $E_i$ sont des foncteurs polynomiaux stricts de type fini non constants sur $k$ et les $\rho_i$ des foncteurs additifs de $\F(\A;k)$.

Alors les assertions suivantes sont équivalentes :
\begin{enumerate}
\item le foncteur $F$ appartient à $\pf_n(\F(\A;k))$ ;
\item tous les foncteurs additifs $\rho_i$ appartiennent à $\pf_n(\F(\A;k))$.
\end{enumerate}
\end{pr}

\begin{proof} Si l'on suppose que les foncteurs $\rho_i$ vérifient la propriété $pf_n$, il en est de même pour les $\tilde{E}_i\circ A_i$ grâce à la proposition~\ref{pr-postcompo-P}. On conclut que $F$ possède lui-même la propriété $pf_n$ par la proposition~\ref{pr-ptens-int}.

Réciproquement, si $F$ possède la propriété $pf_n$, alors il en est de même pour les $\tilde{E}_i\circ\rho_i$ grâce à la proposition~\ref{prptrec2}. La proposition~\ref{pr-recip-postcompo} donne donc la conclusion.
\end{proof}

La démonstration du résultat principal de ce §\,\ref{ssect-fpcor}, le théorème~\ref{th-pol-pfn}, repose sur la proposition~\ref{pr-pfnStein} et le résultat général suivant de changement de base au but.

\begin{pr}\label{pr-excoradd} Supposons que $k\to K$ est une extension de corps commutatifs et que tout foncteur simple à valeurs de dimensions finies de $\add(\A;k)$ possède la propriété $pf_n$ pour un $n\in\mathbb{N}\cup\{\infty\}$. Alors tout foncteur simple $S$ à valeurs de dimensions finies de $\add(\A;K)$ possède la propriété $pf_n$.
\end{pr}

\begin{proof} Comme $S$ est à valeurs de dimensions finies, il existe un foncteur simple à valeurs de dimensions finies $T$ de $\add(\A;k)$ et un épimorphisme $T\otimes K\twoheadrightarrow S$, dont le noyau $N$ est un foncteur fini à valeurs de dimensions finies (cette propriété élémentaire, qui constitue une variante de l'énoncé fonctoriel \cite[proposition~3.11]{DTV}, se montre comme son analogue bien connu pour un module simple de dimension finie sur une algèbre sur un corps). Comme $T$ appartient par hypothèse à $\pf_n(\add(\A;k))$, on déduit de la proposition~\ref{pr-pfnad} que $T\otimes K$ appartient à $\pf_n(\add(\A;K))$. Il suffit donc de montrer que tous les facteurs de composition de $N$ vérifient la propriété $pf_{n-1}$, grâce à la proposition~\ref{pr-pfn-sec}. Comme ceux-ci sont des foncteurs simples à valeurs de dimensions finies de $\add(\A;K)$, on en déduit l'énoncé par récurrence sur $n$.
\end{proof}

\begin{thm}\label{th-pol-pfn} Soient $k$ un corps et $n\in\mathbb{N}\cup\{\infty\}$. Supposons que tous les foncteurs additifs simples à valeurs de dimensions finies de $\add(\A;k)$ appartiennent à $\pf_n(\add(\A;k))$. Supposons également que $k$ est de caractéristique nulle, ou bien que $k$ est de caractéristique $p>0$ et que $\A$ possède la propriété $(T_{n,p})$ de la section \ref{sec-add}. Alors tout foncteur polynomial fini et à valeurs de dimensions finies de $\F(\A;k)$ appartient à $\pf_n(\F(\A;k))$.
\end{thm}

\begin{proof} Supposons d'abord $k$ algébriquement clos. Alors tout foncteur polynomial simple à valeurs de dimensions finies de $\F(\A;k)$ possède une décomposition du type \eqref{eq-decSt} (page~\pageref{eq-decSt}) où les $\rho_i$ sont des foncteurs additifs absolument simples, par \cite[théorème~5.5]{DTV}. Or les hypothèses faites sur $\A$ ainsi que le théorème~\ref{thm-add-Zp} (ou la proposition~\ref{pr-add-Q} en caractéristique nulle) montrent que les foncteurs additifs absolument simples à valeurs de dimensions finies de $\F(\A;k)$ sont dans $\pf_n(\F(\A;k))$. Par la proposition~\ref{pr-pfnStein}, on en déduit que tout foncteur polynomial simple à valeurs de dimensions finies de $\F(\A;k)$ possède la propriété $pf_n$. Comme $\pf_n(\F(\A;k))$ est stable par extensions (proposition~\ref{pr-pfn-sec}), tout foncteur polynomial fini et à valeurs de dimensions finies est dans $\pf_n(\F(\A;k))$.

Le cas général se ramène au cas où $k$ est algébriquement clos en plongeant ce corps dans une clôture algébrique, grâce aux propositions~\ref{pr-excoradd} et~\ref{pr-chbpfn}.
\end{proof}

\begin{rem} On dispose ainsi, dans la catégorie $\F(A,k)$, où $A$ est un anneau et $k$ un corps commutatif, de deux propriétés de finitude homologique pour les foncteurs polynomiaux à valeurs de dimensions finies (qui sont automatiquement finis \cite[lemme~11.10]{DTV}) qui généralisent le lemme de Schwartz \cite[proposition~10.1]{FLS} :
\begin{enumerate}
\item le théorème~\ref{cor-psf}, qui en garantit la propriété $psf_\infty$ (l'hypothèse de valeurs de dimensions finies est même superflue, et la catégorie but peut être arbritraire) ;
\item le théorème~\ref{th-pol-pfn}, qui en montre la propriété $pf_\infty$ dès lors que les deux conditions suivantes sont vérifiées :
\begin{enumerate}
\item tout $(k,A)$-bimodule absolument simple possède la propriété $pf_\infty$ ;
\item si $k$ est de caractéristique $p>0$, la condition $(T_{\infty,p})$ est satisfaite.
\end{enumerate}
\end{enumerate}
En général, un foncteur polynomial (même additif) à valeurs de dimensions finies de $\F(A,k)$ peut ne pas être de présentation finie, car un $(k,A)$-bimodule de dimension finie sur $k$ n'est pas nécessairement de présentation finie.
\end{rem}

\subsection{Foncteurs polynomiaux à valeurs dans les groupes abéliens}\label{ssct-Z}

Dans toute cette section, l'anneau de base est $k=\mathbb{Z}$. On pourrait le remplacer sans peine par un autre anneau principal, mais nous nous focalisons sur les entiers, qui suffisent pour l'application principale que nous avons en vue, le corollaire~\ref{cor-HGLF-ideal} ci-après.

\begin{lm}\label{lm-ptfadd} Si la catégorie $\add(\A;\mathbb{Z})$ est localement noethérienne, alors tout produit tensoriel de foncteurs additifs de type fini sur $\A$ appartient à $\pf_\infty(\F(\A;\mathbb{Z}))$.
\end{lm}

\begin{proof} Par la proposition~\ref{anc-rqc}, tout foncteur additif de type fini appartient à $\pf_\infty(\add(\A;\mathbb{Z}))$. Comme $\A$ vérifie la condition $(T_{\infty,p})$ pour tout nombre premier $p$, par la proposition~\ref{pr-ptt}\,(\ref{pt2}), le théorème~\ref{thm-add-Z} montre que les foncteurs additifs de type fini appartiennent également à $\pf_\infty(\F(\A;\mathbb{Z}))$.

Le passage à un produit tensoriel de foncteurs additifs de type fini s'obtient grâce au corollaire~\ref{cor-ptpf-pasplat}.
\end{proof}

Le théorème~\ref{th-polnlpf} et le corollaire~\ref{cor-polpfi} ci-après permettent d'obtenir dans certains cas la propriété $pf_\infty$ de foncteurs polynomiaux {\em dans la catégorie $\F(\A;\mathbb{Z})$} à partir d'une hypothèse de noethérianité {\em sur les catégories $\pol_d(\A;\mathbb{Z})$}. Ces énoncés constituent un progrès substantiel par rapport à la proposition~\ref{anc-rqc}, car la catégorie $\F(\A;\mathbb{Z})$ est rarement localement noethérienne (\cite[proposition~11.1]{DTV} montre que $\F(A,\mathbb{Z})$ n'est localement noethérienne que si l'anneau $A$ est {\em fini}), tandis que la noethérianité locale de $\pol_d(\A;\mathbb{Z})$ est plus fréquente et plus facile à établir (elle vaut notamment dans le cas important $\A=\mathbf{P}(\mathbb{Z})$).

\begin{thm}\label{th-polnlpf} Soit $d\in\mathbb{N}$. Supposons que la catégorie $\pol_d(\A;\mathbb{Z})$ est localement noethérienne. Alors tout foncteur de type fini et polynomial de degré au plus $d$ de $\F(\A;\mathbb{Z})$ possède la propriété $pf_\infty$.
\end{thm}

\begin{proof} On raisonne par récurrence sur $d$. La propriété est triviale pour $d=0$.

On suppose donc $d>0$, et que les foncteurs polynomiaux de type fini et de degré $<d$ de $\F(\A;\mathbb{Z})$ possèdent la propriété $pf_\infty$. Soit $F$ un foncteur de type fini de $\pol_d(\A;\mathbb{Z})$. Alors le multifoncteur multiadditif $cr_d(F)$ est de type fini, ce qui implique que le foncteur $G:=(\delta_d^* cr_d(F))_{\Si_d}$ de $\pol_d(\A;\mathbb{Z})$, où $\delta_d : \A\to\A^d$ est la diagonale itérée, est également de type fini, donc noethérien par hypothèse sur $\pol_d(\A;\mathbb{Z})$. L'adjonction somme/diagonale fournit un morphisme $G\to F$ dont le noyau $X$ et le conoyau $Y$ sont de degré strictement inférieur à $d$ (cf. \cite{Pira88}). Les foncteurs $X$ (qui est un sous-objet du foncteur noethérien $G$) et $Y$ sont de type fini, l'hypothèse de récurrence montre donc qu'ils appartiennent à $\pf_\infty(\F(\A;\mathbb{Z}))$. Au vu de la proposition~\ref{pr-pfn-sec}, il suffit de montrer que $G$ est dans $\pf_\infty(\F(\A;\mathbb{Z}))$ pour conclure à la même propriété pour $F$.

On vérifie d'abord que le foncteur $\delta_d^* cr_d(F)$ appartient à $\pf_\infty(\F(\A;\mathbb{Z}))$. En effet, $cr_d(F)$ est de type fini dans la sous-catégorie pleine $\add_d(\A;\mathbb{Z})$ des foncteurs de $\F(\A^d;\mathbb{Z})$ qui sont additifs par rapport à chaque variable. Cette catégorie est localement noethérienne, car la catégorie $\Sigma\add_d(\A;\mathbb{Z})\simeq\pol_d(\A;\mathbb{Z})/\pol_{d-1}(\A;\mathbb{Z})$ des multifoncteurs multiadditifs {\em symétriques} (cf. Pirashvili \cite{Pira88}, et \cite[§\,2.6]{DTV}, dont on suit les notations) l'est, comme quotient de la catégorie localement noethérienne $\pol_d(\A;\mathbb{Z})$ par une sous-catégorie localisante (utiliser pour finir le foncteur d'oubli $\Sigma\add_d(\A;\mathbb{Z})\to\add_d(\A;\mathbb{Z})$ et son adjoint, qui sont exacts et fidèles et préservent les objets de type fini). Comme la catégorie $\add_d(\A;\mathbb{Z})$ est engendrée par les foncteurs $\rho_1\boxtimes\dots\boxtimes\rho_d$, où les $\rho_i$ sont des foncteurs de type fini de $\add(\A;\mathbb{Z})$, le lemme~\ref{lm-ptfadd} montre que $\delta_d^* cr_d(F)$ appartient à $\pf_\infty(\F(\A;\mathbb{Z}))$. Pour montrer que c'est aussi le cas de $G=(\delta_d^* cr_d(F))_{\Si_d}$, on considère le complexe barre pour le groupe $\Si_d$ calculant $H_*(\Si_d;\delta_d^* cr_d(F))$ : en chaque degré, c'est une somme directe d'un nombre fini de copies de $\delta_d^* cr_d(F)$, donc un foncteur de $\pf_\infty(\F(\A;\mathbb{Z}))$, et $H_i(\Si_d;\delta_d^* cr_d(F))$ est, en chaque degré $i>0$, polynomiale de degré {\em strictement} inférieur à $d$ (en effet, on observe directement que le $d$-ième effet croisé de ce foncteur est nul). Les $H_i(\Si_d;\delta_d^* cr_d(F))$ sont également de type fini, car les termes du complexe sont noethériens (grâce à l'hypothèse sur $\pol_d(\A;\mathbb{Z})$). Ainsi, l'hypothèse de récurrence entraîne que $H_i(\Si_d;\delta_d^* cr_d(F))$ appartient à $\pf_\infty(\F(\A;\mathbb{Z}))$ pour $i>0$. Le corollaire~\ref{cor-pfn-compl} permet de conclure que $G=H_0(\Si_d;\delta_d^* cr_d(F))$ est aussi dans $\pf_\infty(\F(\A;\mathbb{Z}))$, d'où le théorème.
\end{proof}

\begin{cor}\label{cor-polpfi} Soit $A$ un anneau dont le groupe additif est de type fini. Alors tout foncteur polynomial de type fini de $\F(A,\mathbb{Z})$ possède la propriété $pf_\infty$.
\end{cor}

\begin{proof} L'hypothèse sur $A$ entraîne que les catégories $\pol_d(A,\mathbb{Z})$ sont localement noethériennes, grâce par exemple à \cite[proposition~4.8]{Dja-FM}. La conclusion découle donc du théorème~\ref{th-polnlpf}.
\end{proof}

\subsection{Application à des propriétés de finitude homologique}

\begin{pr}\label{pr-hhtf} Soient $A$ un anneau dont le groupe additif est de type fini et $B : \mathbf{P}(A)^\op\times\mathbf{P}(A)\to\mathbf{Ab}$ un bifoncteur. On suppose que $B$ est polynomial et à valeurs dans les groupes abéliens de type fini. Alors pour tout $n\in\mathbb{N}$, l'homologie de Hochschild $HH_n(\mathbf{P}(A);B)$ est un groupe abélien de type fini.
\end{pr}

\begin{proof} La catégorie des foncteurs polynomiaux de $\F(\mathbf{P}(A)^\op\times\mathbf{P}(A);\mathbb{Z})$ est engendrée par les produits tensoriels extérieurs $F\boxtimes G$, où $F$ (resp. $G$) est un foncteur polynomial de $\F(\mathbf{P}(A)^\op;\mathbb{Z})$ (resp. $\F(\mathbf{P}(A);\mathbb{Z})$) à valeurs de type fini. Un tel produit tensoriel extérieur est noethérien par le lemme~\ref{lm-polfev} ci-dessous. Il s'ensuit que $B$ possède une résolution par un complexe de chaînes dont chaque terme est une somme directe finie de foncteurs de ce type, de sorte qu'il suffit de montrer que $HH_n(\mathbf{P}(A);F\boxtimes G)$ est un groupe abélien de type fini pour tout $n\in\mathbb{N}$ et tous $F, G$ comme précédemment. On note déjà que $\Tor^{\mathbf{P}(A)}_n(F,G)$ est un groupe abélien de type fini, par le théorème~\ref{th-polnlpf} et la proposition~\ref{pr-valext-fini}. Comme $\Tor(F,G)=\Tor^\mathbb{Z}_1\circ(F,G)$ est également un foncteur polynomial à valeurs de type fini, on en déduit le résultat souhaité  en raisonnant par récurrence sur le degré homologique et en utilisant le lien entre homologie de Hochschild et groupes de torsion \cite[(C.10.1)]{Lod} (qui se réduit ici à une suite exacte longue puisque l'anneau $\mathbb{Z}$ est de dimension homologique $1$).
\end{proof}

\begin{lm}\label{lm-polfev} Soit $F : \mathbf{P}(A)\to\E$ un foncteur polynomial dont les valeurs sont des objets noethériens de la catégorie $\E$. Alors $F$ est un objet noethérien de la catégorie $\fct(\mathbf{P}(A),\E)$.
\end{lm}

\begin{proof} Cela découle de ce que le foncteur d'évaluation
$$\pol_d(\mathbf{P}(A),\E)\to\E\qquad F\mapsto F(A^d)$$
est exact et fidèle (cf. \cite[démonstration du lemme~11.10]{DTV})  pour tout $d\in\mathbb{N}$.
\end{proof}

En guise d'application, considérons la propriété (HGLF) suivante, où $A$ désigne un anneau et $\GL(A)$ désigne le groupe linéaire stable $\underset{i\in\mathbb{N}}{\col}\GL_i(A)$ :
\begin{itemize}
\item[(HGLF)] Pour tout $n\in\mathbb{N}$, le groupe abélien $H_n(\GL(A);\mathbb{Z})$ est de type fini.
\end{itemize}

\begin{ex}Quillen \cite{Q-finitude} a montré qu'un anneau d'entiers de corps de nombres vérifie la propriété (HGLF).
\end{ex}

\begin{cor}\label{cor-HGLF-ideal} Soient $A$ un anneau dont le groupe additif est de type fini et $I$ un idéal bilatère nilpotent de $A$. Si $A/I$ vérifie la propriété {\rm (HGLF)}, alors il en est de même pour $A$.
\end{cor}

\begin{proof} En raisonnant par récurrence sur l'indice de nilpotence de $I$, on voit qu'il suffit de traiter le cas où $I^2=0$. Le morphisme de groupes surjectif $\GL(A)\to\GL(A/I)$ a alors un noyau abélien qui s'identifie à $\mathfrak{gl}(I)$, le groupe additif stable $\underset{i\in\mathbb{N}}{\col}\mathfrak{gl}_i(I)$ des matrices carrées à coefficients dans $I$. Grâce à la suite spectrale de Hochschild-Serre
$$E^2_{r,s}=H_r\big(\GL(A/I);H_s(\mathfrak{gl}(I);\mathbb{Z})\big)\Rightarrow H_{r+s}(\GL(A);\mathbb{Z})\,,$$
il suffit de démontrer que le groupe abélien $H_r\big(\GL(A/I);H_s(\mathfrak{gl}(I);\mathbb{Z})\big)$ est de type fini pour chaque $(r,s)\in\mathbb{N}^2$.

Comme $I^2=0$, $I$ est un $(A/I,A/I)$-bimodule, et $\mathfrak{gl}_i(I)\simeq B((A/I)^i,(A/I)^i)$, où $B : \mathbf{P}(A/I)^\op\times\mathbf{P}(A/I)\to\mathbf{Ab}$ désigne le bifoncteur $(U,V)\mapsto \Hom_{A/I}(U,A/I)\underset{A/I}{\otimes}I\underset{A/I}{\otimes}V$ (cet isomorphisme est $\GL_i(A)$-équivariant, où l'action sur $\mathfrak{gl}_i(I)$ est la conjugaison et celle sur $B((A/I)^i,(A/I)^i)$ se fait via la réduction modulo $I$ et la fonctorialité de $B$ en chaque variable). Ainsi, le $\GL(A/I)$-module gradué $H_*(\mathfrak{gl}(I);\mathbb{Z})$ s'identifie à $\underset{i\in\mathbb{N}}{\col}H_*(B((A/I)^i,(A/I)^i);\mathbb{Z})$.

 Comme $B$ est polynomial (de degré au plus $2$) et que l'homologie $H_d(-;\mathbb{Z})$ définit pour tout $d\in\mathbb{N}$ un endofoncteur polynomial (de degré $d$, par la formule de Künneth) des groupes abéliens, les $H_d(-;\mathbb{Z})\circ B : \mathbf{P}(A/I)^\op\times\mathbf{P}(A/I)\to\mathbf{Ab}$ sont des foncteurs polynomiaux (de degré au plus $2d$). Par conséquent, le théorème de Scorichenko \cite{Sco} (voir \cite[théorème~5.6]{DjaR} pour une version publiée) fournit des isomorphismes naturels de groupes abéliens gradués
 $$H_*\big(\GL(A/I);H_s(\mathfrak{gl}(I);\mathbb{Z})\big)\simeq HH_*(\GL(A/I)\times\mathbf{P}(A/I);H_s(-;\mathbb{Z})\circ B)$$
 où $\GL(A/I)$ opère {\em trivialement} sur le bimodule de coefficients. Ainsi, la formule de Künneth et l'hypothèse (HGLF) pour $A/I$  montrent qu'il suffit de vérifier que les groupes abéliens $HH_n(\mathbf{P}(A/I);H_s(-;\mathbb{Z})\circ B)$ sont de type fini. Or $B$, donc aussi les $H_s(-;\mathbb{Z})\circ B$, sont à valeurs de type fini, puisque le groupe additif sous-jacent à $I$ est de type fini. La proposition~\ref{pr-hhtf} donne donc la conclusion.
\end{proof}

\section{La propriété $pf_n$ pour les foncteurs antipolynomiaux}\label{sap}

Dans cette section nous étudions la propriété $pf_n$ pour les foncteurs antipolynomiaux au sens de \cite{DTV} (nous rappelons la définition ci-dessous). Les résultats découlent du transfert de la propriété $pf_n$ par le foncteur de $k$-linéarisation $k[-]$ établie à la proposition~\ref{pr-linearis-pfn}.

\begin{pr}\label{pr-linearis-pfn} Soient $\rho : \A\to\mathbf{Ab}$ un foncteur additif et $n\in\mathbb{N}\cup\{\infty\}$. Si $\rho$ appartient à $\pf_n(\add(\A,\mathbf{Ab}))$, alors $k[\rho]$ appartient à $\pf_n(\F(\A;k))$.
\end{pr}

\begin{proof} La proposition~\ref{pr-pres-concr} montre qu'il existe une résolution projective $P_\bullet\to \rho$ de $\rho$ dans $\add(\A,\mathbf{Ab})$ telle que $P_i$ est un foncteur représentable pour $i\le n$. Par conséquent, il existe un objet simplicial $R_\bullet$ dans $\add(\A,\mathbf{Ab})$ tel que $R_i$ est également représentable pour $i\le n$, et dont l'homotopie est concentrée en degré nul, où elle est isomorphe à $\rho$, grâce à la correspondance de Dold-Kan (voir par exemple \cite[§\,3]{DoP}). En linéarisant, on obtient un objet simplicial $k[R_\bullet]$ de $\F(\A;k)$, dont l'homotopie est concentrée en degré nul, et isomorphe à $k[\rho]$. Cet objet simplicial est projectif de type fini dans $\F(\A;k)$ en chaque degré $\le n$. Considérant le complexe de chaînes associé, et utilisant la proposition~\ref{pr-pres-concr}, on voit que $k[\rho]$ appartient à $\pf_n(\F(\A;k))$.
\end{proof}

\begin{rem}\label{rq-linpf-recip} Pour $n\le 1$, le résultat figure dans \cite[lemme~13.10]{DTV}, qui établit également la réciproque (si l'anneau $k$ est non nul\,\footnote{\cite{DTV} suppose que $k$ est un corps, mais la démonstration s'applique sans changement à tout anneau non nul.}) : si $k[\rho]$ est de type fini, ou de présentation finie, dans $\F(\A;k)$, alors $\rho$ possède la même propriété dans $\add(\A,\mathbf{Ab})$. Lorsque $k=\mathbb{Z}$, on dispose d'un isomorphisme naturel gradué
$$\mathrm{Ext}^*_{\F(\A;\mathbb{Z})}(\mathbb{Z}[\rho],\sigma)\simeq\mathrm{Ext}^*_{\add(\A;\mathbb{Z})}(\rho,\sigma)$$
(qui se déduit du cas du degré nul, immédiat, et de la correspondance de Dold-Kan) qui montre que, pour tout $n$, si $\mathbb{Z}[\rho]$ appartient à $\pf_n(\F(\A;\mathbb{Z}))$, alors $\rho$ appartient à $\pf_n(\add(\A;\mathbb{Z}))$. Nous ignorons si cette réciproque persiste pour la propriété $pf_n$ avec $n>1$ lorsque $k$ est un corps, par exemple.
\end{rem}

\begin{cor}\label{cor-chgtbase-pfn} Soient $\B$ une petite catégorie additive, $\Phi : \B\to\A$ un foncteur additif et $n\in\mathbb{N}\cup\{\infty\}$. Supposons que, pour tout objet $a$ de $\A$, le foncteur additif $\Phi^*\A(a,-)$ appartient à $\pf_n(\add(\B,\mathbf{Ab}))$. Alors le foncteur de précomposition $\Phi^* : \F(\A;k)\to\F(\B;k)$ préserve la propriété $pf_n$.
\end{cor}

\begin{proof} La proposition~\ref{pr-linearis-pfn} montre que $\Phi^*(P_\A^t)$ appartient à $\pf_n(\F(\B;k))$ pour tout objet $t$ de $\A$. La conclusion découle donc de la proposition~\ref{pr-pfn-efonc}.
\end{proof}

Si $k$ est un {\em corps}, on rappelle \cite[§\,4.1]{DTV} que la catégorie additive $\A$ est dite $k$-triviale si les groupes abéliens $\A(x,y)$ sont finis et d'ordre inversible dans $k$ pour tous objets $x$ et $y$ de $\A$. Un idéal $\I$ de $\A$ (c'est-à-dire un sous-foncteur de $\Hom_\A : \A^\op\times\A\to\mathbf{Ab}$) est dit $k$-cotrivial si la catégorie additive $\A/\I$ est $k$-triviale, et un foncteur de $\F(\A;k)$ est dit {\em antipolynomial} s'il se factorise à travers la réduction modulo un idéal $k$-cotrivial de $\A$.

\begin{cor}\label{cor-antipol-pf} Soit $n\in\mathbb{N}\cup\{\infty\}$. Supposons que $k$ est un corps, que, pour tout idéal $k$-cotrivial $\I$ de $\A$, la catégorie $\F(\A/\I;k)$ est localement noethérienne et que pour tout objet $a$ de $\A$, le foncteur $\I(a,-)$ appartient à $\pf_{n-1}(\add(\A,\mathbf{Ab}))$. Alors tout foncteur antipolynomial de type fini de $\F(\A;k)$ est $pf_n$.
\end{cor}

\begin{proof} Le foncteur $\I(a,-)$ est dans $\pf_{n-1}(\add(\A,\mathbf{Ab}))$ si et seulement si $(\A/\I)(a,-)$ est dans $\pf_n(\add(\A,\mathbf{Ab}))$, par la proposition~\ref{pr-pfn-sec} (et l'exemple~\ref{expfinf}(b)). Par conséquent, l'énoncé résulte du corollaire~\ref{cor-chgtbase-pfn}.
\end{proof}

\begin{rem}\label{rq-recip-antipol} Pour $n=1$, la réciproque du corollaire~\ref{cor-antipol-pf} est vraie : si tout foncteur antipolynomial de type fini de $\F(\A;k)$ est de présentation finie, alors le foncteur $\I(a,-)$ est de type fini dans $\add(\A,\mathbf{Ab})$ pour tout idéal $k$-cotrivial $\I$. Cela provient de la réciproque partielle de la proposition~\ref{pr-linearis-pfn} mentionnée à la remarque~\ref{rq-linpf-recip}, à appliquer aux foncteurs antipolynomiaux de type fini $k[(\A/\I)(a,-)]$.
\end{rem}

Rappelons \cite[exemple~4.4]{DTV} qu'un idéal bilatère d'un anneau $A$ est dit \emph{$k$-cotrivial} si $A/I$ est un anneau fini de cardinal inversible dans $k$. Les idéaux $k$-cotriviaux  de $A$ sont en bijection avec les idéaux $k$-cotriviaux de la catégorie $\mathbf{P}(A)$, cette bijection envoyant l'idéal $I$ de $A$ sur l'idéal $\I=I\cdot \Hom_{\mathbf{P}(A)}$ de $\mathbf{P}(A)$. Ainsi $\mathbf{P}(A)/\I\simeq \mathbf{P}(A/I)$ de sorte que la propriété de noethérianité locale est vérifiée pour $\A=\mathbf{P}(A)$, par Putman-Sam-Snowden \cite{PSam,SamSn}. L'hypothèse sur les foncteur $\I(a,-)$ est quant à elle équivalente au caractère $pf_{n-1}$ de $I$ comme $A$-module à droite, d'après le théorème d'Eilenberg-Watts (cf. par exemple \cite[Proposition~2.1]{DTV}). On déduit donc du corollaire \ref{cor-antipol-pf} le résultat suivant.

\begin{cor}\label{cor-antipol-PA} Supposons que $k$ est un corps. Soient $A$ un anneau et $n\in\mathbb{N}\cup\{\infty\}$. Si tout idéal $k$-cotrivial de $A$ est $pf_{n-1}$ comme $A$-module à droite, alors tout foncteur antipolynomial de type fini de $\F(A,k)$ est $pf_n$.
\end{cor}

\begin{ex}\label{ex-antipol-PA} Si $A$ est noethérien à droite, la proposition~\ref{anc-rqc} et le corollaire~\ref{cor-antipol-PA} montrent que tout foncteur antipolynomial de type fini de $\F(A,k)$ est $pf_\infty$.
\end{ex}

\begin{rem}\label{rq-antipol-paspf} En revanche, si $A$ possède un idéal $k$-cotrivial $I$ qui n'est pas de type fini comme idéal à droite (par exemple, si $A$ est un anneau de polynômes en une infinité d'indéterminées sur un corps fini de caractéristique différente de $k$, l'idéal d'augmentation $I$ vérifie ces propriétés), alors il existe un foncteur antipolynomial \emph{simple} (qui est automatiquement à valeurs de dimensions finies) de $\F(A,k)$ qui n'est pas de présentation finie. Cela provient de la stabilité par extensions des foncteurs de présentation finie (cf. proposition~\ref{pr-pfn-sec}), de la remarque~\ref{rq-recip-antipol} et du fait que les foncteurs antipolynomiaux de type fini de $\F(A,k)$ sont finis \cite[corollaire~11.8]{DTV}.
\end{rem}

\section{La propriété $pf_n$ pour les foncteurs finis à valeurs de dimensions finies de $\F(\A;k)$}\label{sfinale}

On suppose dans toute cette section que $k$ est un corps.
Nous déduisons des sections précédentes et des résultats de \cite[§\,4]{DTV} le critère suivant pour qu'un foncteur fini de $\F(\A;k)$ vérifie la propriété $pf_n$.

\begin{thm}\label{th-pfigl} Soient $k$ un corps et $n\in\mathbb{N}\cup\{\infty\}$. On fait les hypothèses suivantes :
\begin{enumerate}
\item\label{hyp1} pour tout idéal $k$-cotrivial $\I\triangleleft\A$, la catégorie $\F(\A/\I;k)$ est localement finie, et pour tout objet $a$ de $\A$, $\I(a,-)$ appartient à $\pf_{n-1}(\add(\A,\mathbf{Ab}))$ ;
\item\label{hyp2} tous les foncteurs additifs simples à valeurs de dimensions finies appartiennent à $\pf_n(\add(\A;k))$ ;
\item\label{hyp3} $k$ est de caractéristique nulle, ou bien de caractéristique $p>0$ et tous les foncteurs ${}_p\A(a,-)$ appartiennent à $\pf_{n-2}(\add(\A;\mathbb{Z}/p))$.
\end{enumerate}

Alors tout foncteur $F$ fini et à valeurs de dimensions finies de $\F(\A;k)$ possède la propriété $pf_n$.
\end{thm}

\begin{proof} D'après \cite[théorème~4.10]{DTV}, il existe un idéal $k$-cotrivial $\I$ de $\A$ et un foncteur $B$ de $\F(\A/\I\times\A;k)$, polynomial par rapport à la deuxième variable, tels que $F$ soit isomorphe à la composée de $B$ et du foncteur canonique $\psi_\I : \A\to\A/\I\times\A$. Via l'isomorphisme canonique $\F(\A/\I\times\A;k)\simeq\fct(\A/\I,\F(\A;k))$, $B$ est à valeurs dans les foncteurs polynomiaux finis prenant des valeurs de dimensions finies. Ces derniers appartiennent à $\pf_n(\F(\A;k))$ par le théorème~\ref{th-pol-pfn}, qu'on peut appliquer grâce aux hypothèses~(\ref{hyp2}) et~(\ref{hyp3}) (cette dernière équivaut à la condition $(T_{n,p})$). Comme les ensembles de morphismes sont finis dans $\A/\I$ et que $\F(\A/\I;k)$ est localement finie, le corollaire~\ref{cor-pfinf-lfcor} permet d'en déduire que $B$ possède également la propriété $pf_n$. Pour conclure, il suffit de noter que le foncteur de précomposition $\psi_\I^* : \F(\A/\I\times\A;k)\to\F(\A;k)$ préserve la propriété $pf_n$, par le corollaire~\ref{cor-chgtbase-pfn} (en effet, $(\A/\I)(a,-)$ appartient à $\pf_n(\add(\A,\mathbf{Ab}))$ pour tout $a\in\mathrm{Ob}\,\A$ grâce à l'hypothèse~(\ref{hyp1})).
\end{proof}

Si $\A=\mathbf{P}(A)$ pour un anneau $A$, on peut déduire des travaux de Putman-Sam-Snowden \cite{PSam,SamSn} que la condition de finitude locale de la première assertion du théorème~\ref{th-pfigl} est vérifiée \cite[proposition~11.7]{DTV}. Par conséquent : 

\begin{cor}\label{cor-pfn-ann} Soient  $k$ un corps, $A$ un anneau et $n\in\mathbb{N}\cup\{\infty\}$. Supposons que :
\begin{enumerate}
\item tout idéal $k$-cotrivial de $A$ vérifie la propriété $pf_{n-1}$ comme $A$-module à droite ;
\item tous les $(k,A)$-bimodules simples de dimension finie sur $k$ vérifient la propriété $pf_n$ (dans la catégorie des $(k,A)$-bimodules) ;
\item $k$ est de caractéristique nulle, ou bien de caractéristique $p>0$ et l'annulateur de $p$ dans $A$ est $pf_{n-2}$ comme $A/p$-module à droite.
\end{enumerate}

Alors tout foncteur fini à valeurs de dimensions finies de $\F(A,k)$ possède la propriété $pf_n$.
\end{cor}

Si l'anneau $A$ est noethérien à droite, les première et troisième conditions du corollaire~\ref{cor-pfn-ann} sont vérifiées (avec $n=\infty$). Si $A\otimes_\mathbb{Z}k$ est noethérien à droite, la deuxième condition est satisfaite. En particulier :

\begin{cor}\label{cor-ttnoeth} Soit $k$ un corps et $A$ un anneau tels que $A$ et $A\otimes_\mathbb{Z}k$ soient noethériens à droite. Alors tout foncteur fini à valeurs de dimensions finies de $\F(A,k)$ vérifie la propriété $pf_\infty$.
\end{cor}

\begin{ex} Si $A$ est un anneau commutatif de type fini, ou bien si $A$ est un anneau noethérien à droite et que $k$ est une extension de type fini de son sous-corps premier, les conditions du corollaire~\ref{cor-ttnoeth} sont satisfaites. 
\end{ex}

\bibliographystyle{amsplain}
\bibliography{DT-bibtran.bib}

\end{document}